\documentclass[12pt,fleqn]{article}
\usepackage{diagbox}
\usepackage{multirow}
\usepackage{amsmath}
\usepackage{amsthm}
\usepackage{amssymb}
\usepackage{amsfonts}
\usepackage{epsf}
\usepackage{graphicx}
\usepackage{hyperref}
\usepackage{color}

\newtheorem{theorem}{Theorem}

\newtheorem{prop}{Proposition}

\theoremstyle{remark}
\newtheorem{rmk}{Remark}
\theoremstyle{conjecture}
\newtheorem{conj}{Conjecture}
\theoremstyle{definition}

 \DeclareMathOperator\Ai{{Ai}}
 \DeclareMathOperator\re{{Re}}
\DeclareMathOperator\im{{Im}} \DeclareMathOperator\sgn{{sgn}} \numberwithin{equation}{section}

\newcommand{\D}{\displaystyle}

\numberwithin{equation}{section}

\newcounter{comment}
\def\a{\alpha}

\def\l{\lambda}
\def\k{\kappa_{0}}

\def\cut{\setminus}

\begin{document}

\title{Singular asymptotics for solutions of the inhomogeneous Painlev\'e II equation}
\date{}
\author{Weiying Hu$^{\dag\ddag}$}

\maketitle

\begin{abstract}
   We consider a family of solutions to the Painlev\'e II equation
   $$
   u''(x)=2u^3(x)+xu(x)-\alpha \qquad \textrm{with } \a \in \mathbb{R} \cut \{0\},
   $$
   which have infinitely many poles on $(-\infty, 0)$. Using Deift-Zhou nonlinear steepest descent method for Riemann-Hilbert problems, we rigorously derive their singular asymptotics as $x \to -\infty$. In the meantime, we extend the existing asymptotic results when $x\to +\infty$ from $\a-\frac{1}{2} \notin \mathbb{Z}$ to any real $\a$. The connection formulas are also obtained.
\end{abstract}


\vspace{2cm}




\vspace{5mm}

\hrule width 65mm

\vspace{2mm}

\begin{description}

\item \hspace*{8mm}$\dag$ College of Mathematics and Stastatics, Shenzhen University, Shenzhen, China. \\
$\ddag$ Department of Mathematics, City University of
Hong Kong, Hong Kong. \\
Email: \texttt{weiyinghu2-c@my.cityu.edu.hk}

\end{description}

\newpage

\section{Introduction and statement of results}
 The Painlev\'{e} II equation (PII)
\begin{equation}\label{PII-def}
   u''(x;\a)=2u^3(x;\a)+xu(x;\a)-\a, \quad \a \in \mathbb{C}
\end{equation}
together with the other five second-order ordinary differential equations, was introduced by Painlev\'e and his colleagues at the begining of the last century. These six equations are of the form $u'' = F (x, u, u')$ with $F$ meromorphic in $x$ and rational in $u$ and $u'$. They satisfy the Painlev\'{e} property: the only movable singularities of a solution $u$ are poles; see more details about the Painlev\'{e} equations and the historical developments in \cite{Fokas2006,Ince1944}.

During the developments of the Painlev\'{e} equations, it has been realized that PII possesses a wild range of important applications in the modern theory of mathematics and physics, such as nonlinear wave motion \cite{AS1977,Miles1978,Rosales1978}, where PII arises as a similarity reduction of the KdV equation; liquid crystal \cite{Clerc2017,Clerc2014,Troy2017}, where PII plays a critical role in light-matter interaction experiments on nematic liquid crystal; random matrix theory \cite{Tracy:Widom1994,Tracy:Widom1999}, where PII appears in the celebrated Tracy-Widom distribution. It is worth mentioning that the Tracy-Widom distribution does not only describe the largest eigenvalue distribution in random matrix ensembles, but also appear in the distribution of the longest increasing subsequence of random permutations \cite{Baik1999},  totally asymmetric simple exclusion
process \cite{Johansson2000}. Although many applications are related to the homogeneous PII, it has been realized that the inhomogeneous PII also plays an important role in random matrix theory and liquid crystal; see \cite{Claeys2008,Clerc2017,Miller2018,Troy2017}

Among the various solutions of PII, those with the boundary condition
\begin{equation*}\label{bc}
u(x;\a)\to 0, \qquad \textrm{as} \,\, x\to +\infty
\end{equation*}
attract the most interests of mathematicians and physicists. 
In \cite{Boutroux1911, Boutroux1914}, Boutroux discovered that PII possesses a family of solutions as follows:
\begin{equation}\label{asy-pos-Fokas}
    u(x;\a)= B(x;\a) + e(x;\a,k), \qquad \textrm{as} \,\, x \to +\infty,
\end{equation}
  where { $B(x;\a)$ has the following full asymptotic expansion
  \begin{equation}\label{B(a;x)}
  B(x;\a)\sim \frac{\alpha}{x}\sum_{n=0}^{\infty} \frac{a_n}{x^{3n}}, \qquad \textrm{as} \,\, x \to +\infty,
  \end{equation}
  and $e(x;\a,k)$ is an exponentially small term, i.e., $e(x ;\a,k )\sim k \Ai (x) \sim k\frac{e^{-\frac{2}{3}x^{3/2}}}{2\sqrt{\pi}x^{1/4}}$ as $x \to + \infty$. In the above expansion, the coefficients $a_n$ are uniquely determined through the following relations:
  \begin{align}\label{a_n-rec-rel}
      a_{n+1}=(3n+1)(3n+2)a_n - 2\a^2 \sum_{\substack{k,l,m =0\\k+l+m=n}}^{n} a_k a_l a_m, \qquad a_0 = 1.
    \end{align}
Note that, this family of solutions $u(x;\a)$ in \eqref{asy-pos-Fokas} depends on the parameter $k$, which appears only in the exponentially small term $e(x; \a,k)$.

In the homogeneous case (that is, $\a=0$), the algebraic term $B(x;\a)$ vanishes. Then, we have
\begin{equation*}
  u(x; 0)\sim k \Ai (x), \qquad \textrm{as } x\to +\infty.
\end{equation*}
It is well-known in the literature that there are three families of solutions depending on the parameter $k \in \mathbb{R}$. When $|k| <1$, there exists a family of oscillatory and pole-free solutions on $\mathbb{R}$, namely the \emph{Ablowitz-Segur(AS) solutions}. The AS solutions were first introduced by Ablowitz and Segur in \cite{AS1976}, where the long time asymptotics of the Kortweg-de Vries equation were studied. When $k=1$, there exists a unique solution which is monotonic and pole-free on $\mathbb{R}$, namely the \emph{Hastings-McLeod(HM) solution}. The HM solution was discovered by Hastings and McLeod in \cite{HM1980}. This solution plays a critical role in the Tracy-Widom distribution \cite{Tracy:Widom1994} and in the asymptotic description of the solution of the KdV equation in the small dispersion limit \cite{Cla:Grava2010}. When $k=-1$, the corresponding solution is obtained through the following simple symmetry relation $ u(x;0) = -u(x;0).$ When $|k| > 1$, there exists a family of \emph{singular solutions}, which have infinitely many poles on the $(-\infty, 0)$. The asymptotics of the singular solutions were first studied by Kapaev in \cite{Kapaev1992}. As a result, we see that there is a critical value $k^* = \pm 1$, where properties of the corresponding solutions change significantly. For more detailed information about this family of solutions for homogeneous PII, we refer to Deift and Zhou \cite{Deift1995},  Bothner and Its \cite{Bothner:Its2012}, Dai and Hu \cite[Sec. 1.1]{Dai:Hu2017}, and references therein.

Inspired from the above results of homogeneous PII, it is natural to expect that similar $k$-dependent results also hold for the solutions $u(x;\a)$ with boundary condition \eqref{asy-pos-Fokas} when $\a \in \mathbb{R} \cut \{0\}$. Regarding this problem, Clarkson \cite{Clarkson2007} made the following conjecture for the inhomogeneous PII when $\a \in \mathbb{Z}\cut \{0\}$. Here, to be in accordance with our notation, we replace $\a$ in \cite{Clarkson2007} by $-\a$.
  \begin{conj}[Clarkson \cite{Clarkson2007}]\label{conjecture}
  Let $k$ be an arbitrary, non-zero real number and $u_{k}(x; n)$ be the solution of PII for $\a = n \in \mathbb{Z}\cut \{0\}$ satisfying \eqref{asy-pos-Fokas}. Then,
  \begin{itemize}
  \item[(a)] there exists a unique $k_n^*$ such that for $k<k_n^*$, $u_k(x;n)$ blows up at a finite $x_1$, with
   \begin{equation}
      u_k(x;n)\sim \sgn(n)(x-x_1)^{-1}, \qquad \textrm{as } x \downarrow x_1;
  \end{equation}
      and for $k>k_n^*$, $u_k(x;n)$ blows up at a finite $x_2$, with
      \begin{equation}
      u_k(x;n)\sim -\sgn(n)(x-x_2)^{-1}, \qquad \textrm{as } x \downarrow x_2.
      \end{equation}
      \item[(b)] for $n>0$, $u_{k_n^*}(x;n)$ is a positive, monotonically decreasing solution, and for $n<0$, $u_{k_n^*}(x;n)$ is a negative, monotonically increasing solution. Furthermore, we have
          \begin{equation}
          u_{k_n^*}(x;n)\sim \begin{cases}\D \frac{n}{x},& \textrm{as } x\to +\infty, \\
          \sgn(n)\sqrt{-\frac{x}{2}}, &\textrm{as } x\to -\infty.
          \end{cases}
          \end{equation}
  \end{itemize}
  \end{conj}

From the above conjecture regarding the family of solutions in \eqref{asy-pos-Fokas}, one can see that there also exists a critical value $k_\a^*$ for the inhomogeneous PII. In the literature, this value has been suggested to be $k_{\a}^* = \pm \cos (\pi \a)$, for all $\a \in \mathbb{R}$, in McCoy and Tang \cite{McCoy:Tang} and Kapaev \cite{Kapaev1992}. Recently, most of parts in the conjecture have been proved rigourously.

When $\a\in (-\frac{1}{2}, \frac{1}{2})$ and $|k|< \cos (\pi \a)$, we have shown that there is a family of oscillatory and pole-free solutions on $\mathbb{R}$, and called them the \emph{AS solutions} as well; see \cite[Thm. 2]{Dai:Hu2017}. If one extends the value of $\a$ to $|\a|> \frac{1}{2}$ and keeps $|k|< |\cos (\pi \a)|$, we have proved that the asymptotic behavior of $ u(x;\a)$ is the same as the AS solutions, but $ [ \, |\alpha| + \frac{1}{2} \, ] $ simple poles will appear on the real axis; see \cite[Thm. 1]{Dai:Hu:poles}. This family of solutions is observed numerically by  Fornberg and Weideman in \cite{Fornberg2014} and named the \emph{quasi-Ablowitz-Segur} (qAS) solutions of PII.

When $|k| = |k_{\a}^*| = | \cos (\pi \a)|$, there also exist monotonic and pole-free solutions on $\mathbb{R}$, which are named the \emph{HM solutions}. The parameters $k$ for the HM solutions are equal to $\sgn(\a) \cos (\pi \a).$ For the case $k = - \sgn(\a) \cos (\pi \a)$, the corresponding solutions are no longer monotonic and may possess finitely many ($ [ \, |\alpha| + \frac{1}{2} \, ] $) poles on the real line; see the numerical plots in  Fornberg and Weideman in \cite{Fornberg2014}. To distinguish these solutions, the  monotonic pole-free solutions, the non-monotonic pole-free solutions, and the solutions possessing poles are named the \emph{primary HM} (pHM) solutions, the \emph{secondary HM} (sHM) solutions, and the \emph{quasi-HM} (qHM) solutions, respectively, in \cite{Fornberg2014}. In recent years, the existence and monotonic properties of these HM solutions have been studied in  \cite{Claeys2008, Clerc2017, Troy2017}. See also \cite{Dai:Hu:poles} for the properties of the qHM solutions. 

Based on the above results, we believe that, when $|k| > |k_{\a}^*| = | \cos (\pi \a)|$, the solutions $u(x;\a)$ with boundary condition \eqref{asy-pos-Fokas} have infinitely many poles on $(-\infty, 0)$. When $x \to -\infty$, they should satisfy similar singular asymptotics as the homogeneous case in Bothner and Its \cite[Thm. 1]{Bothner:Its2012}. However, to the best of our knowledge, this has not been established in the literature. Moreover, the asymptotics the \eqref{asy-pos-Fokas} as $x \to +\infty$ have only been rigorously proved for  $\a-\frac{1}{2}\notin \mathbb{Z}$ in Its and Kapaev \cite{Its2003}. They didn't cover the case $\a-\frac{1}{2}\in \mathbb{Z}$ due to some technical reasons. In addition, the connection formulas describing the relation between asymptotics as $x \to \pm \infty$ have not been rigorously justified when $\a \in \mathbb{R} \setminus \{0\}$.

The purpose of the present paper is to study the asymptotics for the inhomogeneous PII when the parameter $|k| > | \cos (\pi \a)|$. We first extend the existing asymptotics \eqref{asy-pos-Fokas} when $x\to +\infty$ from $\a-\frac{1}{2} \notin \mathbb{Z}$ to any real $\a$. Then, we derive the singular asymptotics as $x \to -\infty$ and prove the connection formulas rigorously.

}
%

\subsection{Our results}\label{sec:results}
In this paper, we will prove the following theorem:

\begin{theorem}\label{main-thm}
   { Given $\a \in \mathbb{R}$ and $k \in \mathbb{R}$ with $ |k| > |\cos(\pi \a)|$,} there exists a set of real-valued solutions $u(x;\a)$ of PII in \eqref{PII-def} possessing the following properties:
    \begin{itemize}

    \item[(a)] $u(x;\a)$ satisfies the asymptotics as $x \to +\infty$:
    \begin{equation}\label{asy-pos}
    u(x;\a)= B(x;\a) + k\Ai(x)(1+O(x^{-\frac{3}{4}})),
    \end{equation}
    where $\Ai(x)$ is the Airy function and $B(x;\a)$ has an asymptotic expansion \begin{equation}\label{B(a;x)}
  B(x;\a)\sim \frac{\alpha}{x}\sum_{n=0}^{\infty} \frac{a_n}{x^{3n}}, \qquad \textrm{as} \,\, x \to +\infty,
  \end{equation}
  with the coefficients $a_n$ uniquely determined in \eqref{a_n-rec-rel}.

  \item[(b)] $u(x;\a)$ satisfies the asymptotics as $x \to -\infty$:
    \begin{equation}\label{asy-neg}
    u(x;\a)=\frac{\sqrt{-x}}{\sin\{\frac{2}{3}(-x)^{\frac{3}{2}}+\frac{3}{4}d^2\ln(-x) + \phi\}+O((-x)^{-\frac{3}{2}})}+O((-x)^{-1}),
    \end{equation}
    uniformly for $x$ bounded away from the zeros of the denominator in \eqref{asy-neg}. Besides, the constants $d$ and $\phi$ in \eqref{asy-neg} are related to the parameter $k$ in \eqref{asy-pos} through the connection formulas as follows:
    \begin{equation}\label{d-k}
    d(k) = \frac{1}{\sqrt{\pi}} \sqrt{\ln(k^2 - \cos^2(\pi \a))},
    \end{equation}
    \begin{equation}\label{phi-k}
    \phi(k) = \frac{3\ln 2}{2} d^2  -\arg \Gamma{\biggr(\frac{1}{2}id^2+\frac{1}{2}}\biggr)-\arg (-\sin (\pi \a) -ki).
    \end{equation}
    \end{itemize}
\end{theorem}

Since the solutions to PII satisfy the symmetry connection
\begin{equation} \label{symmetry-relation}
 u(x;-\a) = -u(x;\a),
\end{equation}
we assume $\a \geq 0$ throughout the rest of this paper. Combining previous results in the literature and the above theorem, we summarize solutions of PII equation with the boundary behavior \eqref{asy-pos-Fokas} in the following Table \ref{solutions types}, depending on $\a$ and $k$.
\begin{table}[h]
{\centering
\begin{tabular}{|c|c|c|c|c|}
\hline
\diagbox[width=3cm]{$\a$}{$u(x;\a)$}{$k$} & $|k|<|\cos (\pi \a)|$  & $k=\cos (\pi \a) $ & $k=-\cos (\pi \a)$ & $|k|>|\cos (\pi \a)|$ \\
\hline
$0$ & AS & 
pHM& pHM & singular\\
\hline
$(0, \frac{1}{2})$ & AS & pHM & sHM & singular\\
\hline
$(n-\frac{1}{2}, n+\frac{1}{2})$ & qAS & pHM & qHM & singular \\
\hline
$n \pm \frac{1}{2}$ & D.N.E. & \multicolumn{2}{|c|}{pHM}  & singular\\
\hline
\end{tabular}\caption{PII solutions with the boundary condition \eqref{asy-pos-Fokas} when $\a \geq 0$, $k\in \mathbb{R}$ and $n\in \mathbb{N}$. Here ``D.N.E." stands for ``does not exist".}\label{solutions types}}
\end{table}

The asymptotic formulas \eqref{asy-pos}-\eqref{phi-k} were first obtained in 1992 by Kapaev \cite{Kapaev1992} with the help of the isomonodromy method. Some years later, rigorous proofs were derived by applying Deift-Zhou nonlinear steepest descent method: Its and Kapaev \cite{Its2003} proved the asymptotics \eqref{asy-pos} for any $\a- 1/2 \notin \mathbb{Z}$ when $x\to +\infty$; Bothner and Its \cite{Bothner:Its2012} proved the asymptotics \eqref{asy-neg} as $x\to-\infty$ and the connection formulas \eqref{d-k}-\eqref{phi-k} for the homogeneous case ($\a=0$).

Comparing with the isomonodromy method, one advantage of the nonlinear steepest descent method is that no prior assumption about the behavior of the solution is needed, which makes the asymptotics derived by this method rigorous. The later method was first introduced by Deift and Zhou \cite{Deift1993} in 1993 and it has been successfully applied to solve asymptotic problems in many fields; see \cite{Buck:Miller2015,Cla:Tamara2010,Claeys2008,Dei:Kri:McL:Ven:Zhou1999-2,Deift1995,Its2003,Kri:McL1999,Xu:Dai:Zhao2014}.

\subsection{Riemann-Hilbert problem for PII}\label{sec:RHP-PII}

%
%

We will prove Theorem \ref{main-thm} by using Deift-Zhou nonlinear steepest descent method for Riemann-Hilbert (RH) problems. Let us first introduce the RH problem associated with PII for all $\a >0$ as follows; see \cite{Claeys2008,Its2003}. Corresponding results for $\a <0$ can be obtained through the symmetry relation \eqref{symmetry-relation}.

\medskip

\noindent\textbf{RH problem for $\Psi_\a(\l)$:} \\
We seek a $2 \times 2$ matrix-valued function $\Psi_\a(\l)$ satisfying the following properties.
\begin{itemize}
\item[(a)] $\Psi_\a(\l)$ is analytic for $\l \in \mathbb{C} \cut \Sigma$. Here $\Sigma =\D\cup_{k= 1}^{4} \gamma_k$ with $\gamma_1=\{\l \in \mathbb{C}: \textrm{arg}\, \l = \frac{\pi}{6}\}$, $\gamma_2=\{\l \in \mathbb{C}: \textrm{arg}\, \l = \frac{5\pi}{6}\}$ , $\gamma_3=\{\l \in \mathbb{C}: \textrm{arg}\, \l = -\frac{5\pi}{6}\}$ and $\gamma_4=\{\l \in \mathbb{C}: \textrm{arg}\, \l = -\frac{\pi}{6}\}$ are four rays oriented to infinity; see Figure \ref{contour-Arno}.

    \begin{figure}[h]
\centering
  \includegraphics[width=10cm]{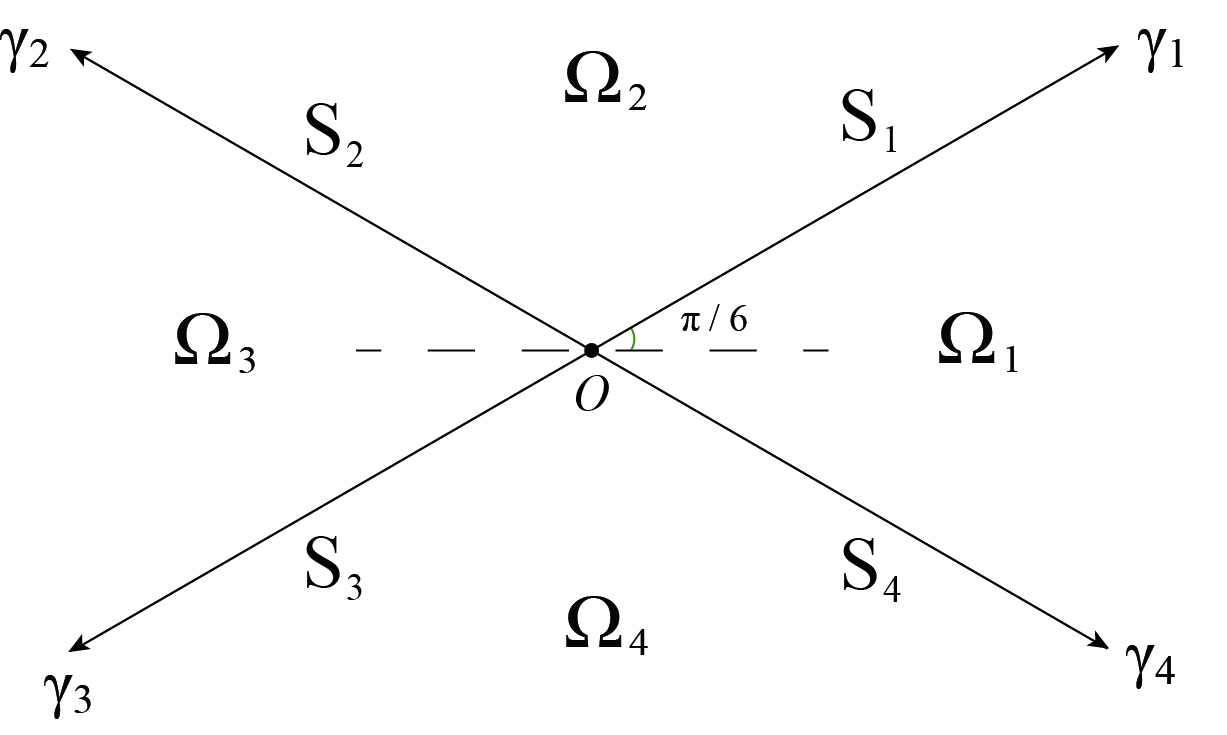}\\
  \caption{The contour $\Sigma$ and corresponding jump matrices.}\label{contour-Arno}
\end{figure}

\item[(b)] Let $\Psi_{\a,\pm}(\l)$ denote the limits of $\Psi_{\a}(\l)$ as $\l$ tends to the jump contour from left and right hand sides, respectively. They satisfy the following jump relations:
\begin{equation} \label{Psi+Psi-}
\Psi_{\a, +}(\l)= \Psi_{\a, -} (\l) S_k, \qquad \textrm{for } \l \in \gamma_k,
\end{equation}
with
\begin{eqnarray} \label{Psi-jumps}
 &S_1& = \left(
  \begin{array}{cc}
  1 & 0\\
    s_1 & 1
  \end{array}
    \right),  \quad S_2 = \left(
  \begin{array}{cc}
  1 & 0\\
    s_3 & 1
  \end{array}
    \right), \\
     &S_3& = \left(
  \begin{array}{cc}
  1 & -s_1\\
    0 & 1
  \end{array}
    \right),  \quad S_4 = \left(
  \begin{array}{cc}
  1 & -s_3\\
    0 & 1
  \end{array}
    \right), \label{Psi-jumps2}
\end{eqnarray}
and
 \begin{equation}\label{stokes-real-as}
 s_1=\bar{s}_3 = -\sin (\pi \a) - k i, \quad \textrm{with } k\in \mathbb{R} \textrm{ and } |k|> |\cos \pi \a|.
 \end{equation}

 \item[(c)] As $\l \to \infty$, $\Psi_\a(\l)$ satisfies the asymptotics:
 \begin{equation}\label{Psi-infty-asy}
  \Psi_\a(\l)e^{\theta (\l)\sigma_3}  =I+ \frac{\Psi_{1}(x)}{\l}+O(\l^{-2}) , 
 \end{equation}
 with
 \begin{equation} \label{theta-def}
   \theta(\l)= \frac{4}{3}i\l^3 + ix\l,\qquad \sigma_3 = \left(\begin{array}{cc}
  1 & 0\\
    0 & -1
  \end{array}\right).
 \end{equation}
 Note that the function $\Psi_1(x)$ is unknown at present.

 \item[(d)] At $\l =0$, $\Psi_\a(\l)$ satisfies the following behaviors:
\begin{eqnarray}
\Psi_{\a}(\l)=O\left(\begin{matrix}
  |\l|^{-\a} &  |\l|^{-\a}\\
   |\l|^{-\a} & |\l|^{-\a}
  \end{matrix}\right), \quad \l \in \Omega_{1,3},\label{asy-psi-origin-1}\\
  \Psi_{\a}(\l)\left(\begin{matrix}
  1 & \frac{s_3 +i e^{-\pi i \a}}{1-s_1s_3}\\
   0 & 1
  \end{matrix}\right) =O\left(\begin{matrix}
  |\l|^{-\a} &  |\l|^{\a}\\
   |\l|^{-\a} & |\l|^{\a}
  \end{matrix}\right), \quad \l \in \Omega_{2},\label{asy-psi-origin-2}\\
  \Psi_{\a}(\l)\left(\begin{matrix}
  1 & 0\\
  -\frac{s_3 +i e^{-\pi i \a}}{1-s_1s_3} & 1
  \end{matrix}\right) =O\left(\begin{matrix}
  |\l|^{\a} &  |\l|^{-\a}\\
   |\l|^{\a} & |\l|^{-\a}
  \end{matrix}\right), \quad \l \in \Omega_{4},\label{asy-psi-origin-3}
\end{eqnarray}
where the branch cut of $\l^{\a}$ is chosen arbitrarily.
\end{itemize}

From \cite{Bolibruch2004, Bothner:Its2012}, we know that the above RH problem is meromorphically solvable in terms of $x$ and its solution is related to the meromorphic solution of PII equation \eqref{PII-def} by the following connection:
\begin{equation}\label{PII-RHP-rel}
  u (x; \a) =2 \biggl( \Psi_{1}(x) \biggr)_{12}.
\end{equation}

\begin{rmk}
By Liouville's theorem, one can easily verify that, if there exists a solution to the above RH problem, it must be unique.
\end{rmk}


\begin{rmk} \label{rmk-real-reduction}
According to \cite[Remark 11.6]{Fokas2006}, a sufficient condition for solution $u(x;\a)$ of PII to be real for real $x$ is
 \begin{equation}\label{iso-data-con}
     s_1+s_3  = - 2 \sin \pi\a, \qquad s_1 = \bar{s}_3,
\end{equation}
where $s_k$ are the  Stokes multipliers in \eqref{Psi-jumps} and \eqref{Psi-jumps2}. Of course, the values in \eqref{stokes-real-as} satisfy the above requirement. Moreover, since we are focusing on the case $|k|>|\cos(\pi \a)|$ in Theorem \ref{main-thm}, we have $|s_1|=|s_3|>1$ and $1-s_1s_3=k^2-\cos^2(\pi \a)<0$. This implies that the behaviors in \eqref{asy-psi-origin-2} and \eqref{asy-psi-origin-3} are well-defined.

Note that, for AS (or qAS) and HM (or qHM) solutions, the Stokes multipliers are chosen to satisfy the conditions $|s_1|=|s_3|<1$ and $|s_1|=|s_3|=1$, respectively; see \cite{Claeys2008,Dai:Hu2017}.
Then, from a slightly different point of view, the three families of PII solutions with boundary condition \eqref{asy-pos-Fokas} in Table \ref{solutions types} can also be classified based on the sign of $1-s_1s_3$.
\end{rmk}

\begin{rmk}\label{rmk-half-int-singular}
 At $\l=0$, we claim that the RH problem for $\Psi_{\a}(\l)$ does work for $\a - 1/2 \in \mathbb{N}$. It is true that a logarithmic singularity will appear at the origin when $\a -1/2\in \mathbb{N}$. An interesting phenomenon is that, when we are considering the asymptotic behavior of $\Psi_\a(\l)$ near $\l =0$, the contribution of the logarithmic singularity is absorbed by the algebraic terms in \eqref{asy-psi-origin-1}-\eqref{asy-psi-origin-3}. This issue has been discussed carefully in Claeys et al. in \cite[Prop. 2.3]{Claeys2008}. Similar results also hold in our case; see the detailed description in Proposition \ref{prop-M} regarding of the function $M$ in the following Section.
\end{rmk}

The rest of this paper is arranged as follows. We will first introduce a model RH problem and give its explicit solution in Section \ref{modl-RHP}. It plays an important role in extending the asymptotics \eqref{asy-pos} of Its and Kapaev \cite{Its2003} from $\a - 1/2 \notin \mathbb{Z}$ to any real $\a$ as $x\to +\infty$ in Section \ref{sec:extension}. It is also used to construct the local parametrix at origin in the nonlinear steepest descent analysis for $x\to -\infty$ in Section \ref{sec-non-anal}. Finally, in the last Section, we finish the proof of Theorem \ref{main-thm}.

\section{A model RH problem}\label{modl-RHP}

In this section, we first introduce a model RH problem for $M(\eta)$ and give an explicit solution for any $\a>0$.
\subsection{RH problem for $M$}

\begin{itemize}

\item[(a)] $M(\eta)$ is analytic when $\eta \in \mathbb{C} \cut \Sigma_M$; Here $\Sigma_M =\D\cup_{k= 1}^{4} \Gamma_k$ with $\Gamma_1=\{\eta \in \mathbb{C}: \textrm{arg}\, \eta = \frac{\pi}{6}\}$, $\Gamma_2=\{\eta \in \mathbb{C}: \textrm{arg}\, \eta =- \frac{7\pi}{6}\}$ , $\Gamma_3=\{\eta \in \mathbb{C}: \textrm{arg}\, \eta = -\frac{5\pi}{6}\}$ and $\Gamma_4=\{\eta \in \mathbb{C}: \textrm{arg}\, \eta = -\frac{\pi}{6}\}$ are four rays oriented to infinity; see Figure \ref{contour-M}.

\begin{figure}[h]
  \centering
  \includegraphics[width=8cm]{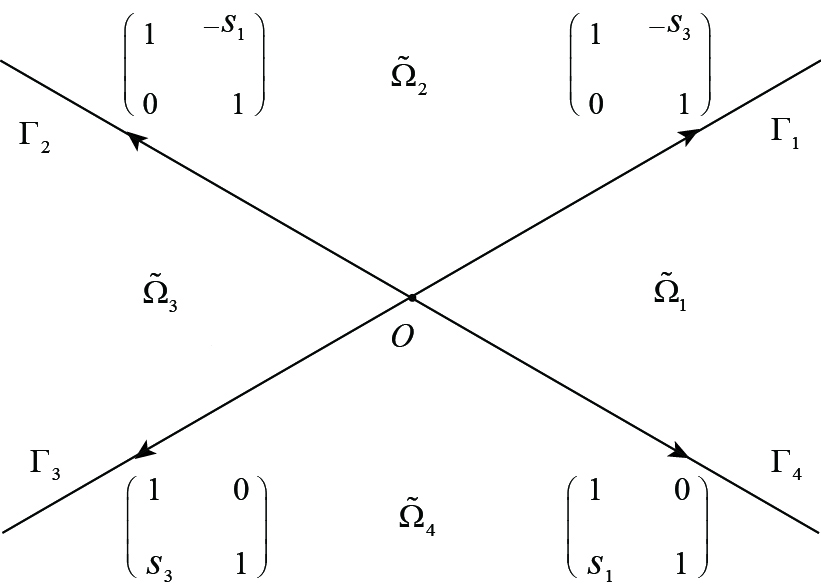}
  \caption{The contour $\Sigma_M$ with corresponding jump matrices $J_M$.}\label{contour-M}
\end{figure}

\item[(b)]On $\Sigma_M$, $M(\eta)$ satisfies the following jump relations:
\begin{equation}\label{M+ M-}
 M_{+}(\eta) = M_{-}(\eta) J_M,
\end{equation}
where the jump matrices $J_M$ are given in Figure \ref{contour-M}. Here $s_1$ and $s_3$ are required to satisfy the first condition in \eqref{iso-data-con}, i.e.,  $s_1+s_3 = -2 \sin (\pi \a)$.

\item[(c)] $M(\eta)$ has the following asymptotics as $\eta \to \infty$:
\begin{equation} \label{M-large-eta}
M(\eta) = \left(I+O\left(\frac{1}{\eta}\right)\right)e^{i\eta\sigma_3}.
\end{equation}
\item[(d)]$M(\eta)$ has the following singularity at $\eta = 0$:
 \begin{eqnarray}
M(\eta)=&O\left(\begin{matrix}
  |\eta|^{-\a} &  |\eta|^{-\a}\\
   |\eta|^{-\a} & |\eta|^{-\a}
  \end{matrix}\right), \quad \eta \in \tilde{\Omega}_{1,3},\label{M-origin-asy-1}\\
 M(\eta)\left(\begin{matrix}
  1 & s_3 +i e^{-\pi i \a}\\
   0 & 1
  \end{matrix}\right) =&O\left(\begin{matrix}
  |\eta|^{-\a} &  |\eta|^{\a}\\
   |\eta|^{-\a} & |\eta|^{\a}
  \end{matrix}\right), \quad \eta \in \tilde{\Omega}_{2},\label{M-origin-asy-2}\\
 M(\eta)\left(\begin{matrix}
  1 & 0\\
  -s_3 -i e^{-\pi i \a} & 1
  \end{matrix}\right) =&O\left(\begin{matrix}
  |\eta|^{\a} &  |\eta|^{-\a}\\
   |\eta|^{\a} & |\eta|^{-\a}
  \end{matrix}\right), \quad \eta \in \tilde{\Omega}_{4},\label{M-origin-asy-3}
\end{eqnarray}
where the branch cuts of $\eta^{\pm\a}$ are chosen along $\Gamma_2$.
\end{itemize}

 The above RH problem is motivated from the problem in \cite[P. 609-610]{Claeys2008}, which is associated with the HM solutions for PII. Note that the coefficients $\left(\begin{matrix}
  1 & 0\\
  s_3 +i e^{-\pi i \a} & 1
  \end{matrix}\right)$ and $\left(\begin{matrix}
  1 & 0\\
  -s_3 -i e^{-\pi i \a} & 1
  \end{matrix}\right)$ are from the original RH problem for $\Psi_{\a}$.

\begin{prop}\label{prop-M}
When $\a\geq 0$, the unique solution to the above RH problem for $M$ is given by
\begin{equation}\label{M-expression}
M(\eta)=\begin{cases}\D M_1(\eta) , & \quad \quad \eta \in \tilde{\Omega}_1\\
M_1(\eta)\left(\begin{matrix}
1 & -s_3\\
0 & 1
\end{matrix}\right) ,& \quad \quad \eta \in \tilde{\Omega}_2\\
M_2(\eta),& \quad \quad \eta \in \tilde{\Omega}_3\\
M_2(\eta)\left(\begin{matrix}
1 & 0\\
s_3 & 1
\end{matrix}\right) ,& \quad \quad \eta \in \tilde{\Omega}_4,
\end{cases}
\end{equation}
where $M_k(\eta)$, $k = 1,2$, are defined in terms of Hankel functions $H^{(1,\,2)}_{\a -\frac{1}{2}}$ and $H^{(1,\,2)}_{\a+\frac{1}{2}}$ as follows:
\begin{equation}\label{M1}
M_1(\eta)=\frac{\sqrt{\pi}}{2\sqrt{2}}\left(\begin{matrix}
1 & i\\
i & 1
\end{matrix}\right) \left(\begin{matrix}i\eta^{\frac{1}{2}} H^{(1)}_{\a +\frac{1}{2}}(\eta) & -\eta^{\frac{1}{2}}H^{(2)}_{\a +\frac{1}{2}}(\eta)\\
-i\eta^{\frac{1}{2}} H^{(1)}_{\a -\frac{1}{2}}(\eta) & \eta^{\frac{1}{2}}H^{(2)}_{\a -\frac{1}{2}}(\eta)\end{matrix}\right) e^{\frac{\a}{2}\pi i\sigma_3}
\end{equation}
and
\begin{equation}\label{M2}
M_2(\eta)=\frac{\sqrt{\pi}}{2\sqrt{2}}\left(\begin{matrix}
1 & i\\
i & 1
\end{matrix}\right) \left(\begin{matrix}i \eta^{\frac{1}{2}}H^{(2)}_{\a +\frac{1}{2}}(e^{\pi i}\eta) & \eta^{\frac{1}{2}}H^{(1)}_{\a +\frac{1}{2}}(e^{\pi i}\eta)\\
i\eta^{\frac{1}{2}} H^{(2)}_{\a -\frac{1}{2}}(e^{\pi i}\eta) & \eta^{\frac{1}{2}}H^{(1)}_{\a -\frac{1}{2}}(e^{\pi i}\eta)\end{matrix}\right) e^{-\frac{\a+1 }{2}\pi i\sigma_3}.
\end{equation}
Moreover, as $\a - \frac{1}{2} \in \mathbb{N}$, $M(\eta)$ has a logarithmic singularity at the origin.
\end{prop}

\begin{proof}
The uniqueness of the solution to the RH problem for $M$ is easy to check by applying the Liouville's theorem. Next, we prove that the function defined in Proposition \ref{prop-M} solve the RH problem for $M$.

According to the properties of $H^{(1,\,2)}_{\a -\frac{1}{2}}$ and $H^{(1,\,2)}_{\a+\frac{1}{2}}$ in \cite[Eq. (10.11.3)-(10.11.4)]{Olver:NIST}, we have the following formulas for $m\in \mathbb{Z}$ and all $\eta\in \mathbb{C}$,
\begin{eqnarray}
 \hspace{-15pt} \sin\left(\nu\pi\right){H^{(1)}_{\nu}}\left(\eta e^{m\pi i}\right)&=&-\sin\left((m-1%
)\nu\pi\right){H^{(1)}_{\nu}}\left(\eta\right)-e^{-\nu\pi i}\sin\left(m\nu\pi%
\right){H^{(2)}_{\nu}}\left(\eta\right) \\
 \hspace{-15pt} \sin\left(\nu\pi\right){H^{(2)}_{\nu}}\left(\eta e^{m\pi i}\right)&=&e^{\nu\pi i}%
\sin\left(m\nu\pi\right){H^{(1)}_{\nu}}\left(\eta\right)+\sin\left((m+1)\nu\pi%
\right){H^{(2)}_{\nu}}\left(\eta\right).
\end{eqnarray}
Then, we obtain the following relations between $M_1(\eta)$ and $M_2(\eta)$:
\begin{equation}
M_1(\eta)=M_2(\eta)\left(\begin{matrix}
1 & 0\\
-2\sin (\pi \a) & 1
\end{matrix}\right)=M_2(\eta)\left(\begin{matrix}
1 & 0\\
s_1 & 1
\end{matrix}\right)\left(\begin{matrix}
1 & 0\\
s_3 & 1
\end{matrix}\right),
\end{equation}
\begin{equation}
M_2(e^{-2\pi i}\eta)=M_1(\eta)\left(\begin{matrix}
1 & 2\sin (\pi \a)\\
0 & 1
\end{matrix}\right)=M_1(\eta)\left(\begin{matrix}
1 & -s_1\\
0 & 1
\end{matrix}\right)\left(\begin{matrix}
1 & -s_3\\
0 & 1
\end{matrix}\right).
\end{equation}
Thus, the function defined in Proposition \ref{prop-M} satisfies the jump conditions in the RH problem for $M$. Next, we check the behaviors  as $\eta\to \infty$ and $\eta \to 0$.
\begin{itemize}
\item Behaviors as $\eta \to \infty$:\\ According to the following asymptotics in \cite[Eq. (10.17.5)-(10.17.6)]{Olver:NIST}:
$$H^{(1)}_{\nu}\left(\eta\right)\sim\left(\frac{2}{\pi \eta}\right)^{\frac{1}{2}}e^{i(\eta-\frac{1}{2}\nu\pi-
\frac{1}{4}\pi)}(1+O(\eta^{-1})), \qquad  \eta \to \infty, \quad \arg{\eta} \in (-\pi, 2\pi)$$
and
$$ H^{(2)}_{\nu}\left(\eta\right)\sim\left(\frac{2}{\pi \eta}\right)^{\frac{1}{2}}e^{-i(\eta-\frac{1}{2}\nu\pi-
\frac{1}{4}\pi)}(1+O(\eta^{-1})), \qquad  \eta \to \infty,\quad \arg{\eta} \in (-2\pi, \pi),$$
it is easy to see that
\begin{eqnarray}
M_1(\eta)&=&(I+O(\eta^{-1}))e^{i\eta \sigma_3}, \quad \eta\to \infty, \quad \arg{\eta} \in (-\pi, \pi),\\
M_2(\eta)&=&(I+O(\eta^{-1}))e^{i\eta \sigma_3}, \quad \eta\to \infty, \quad \arg{\eta} \in (-2\pi, 0).
\end{eqnarray}

Then, for $\eta \in \tilde{\Omega}_{1,3}$, we have $M_1(\eta)=(I+O(\eta^{-1}))e^{i\eta \sigma_3}$ as $\eta\to \infty$.
For $\eta \in \tilde{\Omega}_2$, we get
\begin{eqnarray}
M_{1}(\eta)\left(\begin{matrix}
  1 & -s_3\\
   0 & 1
  \end{matrix}\right)&=&(I+O(\eta^{-1}))e^{i\eta \sigma_3}\left(\begin{matrix}
  1 & -s_3\\
   0 & 1
  \end{matrix}\right)\nonumber\\
  &=&(I+O(\eta^{-1}))\left(\begin{matrix}
  1 & -s_3 e^{2i\eta}\\
   0 & 1
  \end{matrix}\right)e^{i\eta \sigma_3}\nonumber\\
  &=&(I+O(\eta^{-1}))e^{i\eta \sigma_3}, \qquad \textrm{as } \eta \to \infty.
  \end{eqnarray}
  Here we use the fact that when $\eta \in \tilde{\Omega}_2$ and $\eta\to \infty$, $|e^{2i\eta}|=e^{-2\im{\eta}}$ is exponentially small. Similar, we also have $M_{2}(\eta)\left(\begin{matrix}
  1 & 0\\
   s_3 & 1
  \end{matrix}\right)=(I+O(\eta^{-1}))e^{i\eta \sigma_3}$ as $\eta \to \infty$ for $\eta \in \tilde{\Omega}_4$. This establishes the asymptotics in \eqref{M-large-eta}.

  \item Behaviors as $\eta \to 0$:\\
  Similarly, according to the following asymptotics in \cite[Eq. (10.7.2)\&(10.7.7)]{Olver:NIST}:
   \begin{equation}
   {H^{(1)}_{\nu}}\left(\eta\right)\sim-{H^{(2)}_{\nu}}\left(\eta\right)\sim-(i/\pi)%
\Gamma\left(\nu\right)(\frac{1}{2}\eta)^{-\nu}, \qquad \eta \to 0
 \end{equation}
 for $\re{\nu}>0$, and
 \begin{equation}\label{H0-origin}
{H^{(1)}_{0}}\left(\eta\right)\sim-{H^{(2)}_{0}}\left(\eta\right)\sim(2i/\pi)\ln \eta, \qquad \eta \to 0,
\end{equation}
we have the following estimations of $M$ near the origin.
\begin{itemize}
\item[(i)] When $\a \geq \frac{1}{2}$: now $\a\pm \frac{1}{2}\geq 0$.\\
 \noindent \emph{- For $\eta \in \tilde{\Omega}_1$}:
\begin{eqnarray}\label{asy-0-M1}
M(\eta)&=&M_1(\eta)\nonumber\\
&\sim & \frac{\eta^{\frac{1}{2}}}{2\sqrt{2\pi}}\left(\begin{matrix}
1 & i\\
i & 1
\end{matrix}\right) \left(\begin{matrix} \Gamma(\a+\frac{1}{2})(\frac{1}{2}\eta)^{-\a-\frac{1}{2}} & -i \Gamma(\a+\frac{1}{2})(\frac{1}{2}\eta)^{-\a-\frac{1}{2}}\\
- \Gamma(\a-\frac{1}{2})(\frac{1}{2}\eta)^{-\a+\frac{1}{2}} & i \Gamma(\a-\frac{1}{2})(\frac{1}{2}\eta)^{-\a+\frac{1}{2}}\end{matrix}\right) e^{\frac{\a}{2}\pi i\sigma_3}\nonumber \\
&=& O\left(\begin{matrix}
  |\eta|^{-\a} &  |\eta|^{-\a}\\
   |\eta|^{-\a} & |\eta|^{-\a}
  \end{matrix}\right)\qquad \textrm{as } \eta \to 0.
\end{eqnarray}
\noindent \emph{- For $\eta \in \tilde{\Omega}_2$}:
\begin{eqnarray}\label{asy-0-M1}
 && M(\eta)\left(\begin{matrix}
  1 & s_3 +i e^{-\pi i \a}\\
   0 & 1
  \end{matrix}\right)
  =M_1(\eta)\left(\begin{matrix}
  1 & i e^{-\pi i \a}\\
   0 & 1
  \end{matrix}\right)\nonumber\\
& & \  = \frac{\sqrt{\pi}\eta^{\frac{1}{2}}}{2\sqrt{2}}\left(\begin{matrix}e^{\frac{\a}{2}\pi i} \left(iH^{(1)}_{\a +\frac{1}{2}}(\eta) + H^{(1)}_{\a -\frac{1}{2}}(\eta)\right) & e^{-\frac{\a}{2}\pi i}\left(-2J_{\a +\frac{1}{2}}(\eta)+2iJ_{\a -\frac{1}{2}}(\eta)\right)\\
-e^{\frac{\a}{2}\pi i} \left(H^{(1)}_{\a +\frac{1}{2}}(\eta)+iH^{(1)}_{\a -\frac{1}{2}}(\eta)\right) & e^{-\frac{\a}{2}\pi i}\left(-2iJ_{\a +\frac{1}{2}}(\eta)+2J_{\a -\frac{1}{2}}(\eta)\right)\end{matrix}\right)\nonumber \\
& & \  = O\left(\begin{matrix}
  |\eta|^{-\a} &  |\eta|^{\a}\\
   |\eta|^{-\a} & |\eta|^{\a}
  \end{matrix}\right), \qquad \textrm{as } \eta \to 0,
\end{eqnarray}
in which we have used the fact that the logarithmic singularity $\ln \eta$ can be absorbed by the algebraic singularity $\eta^{-\a}$ at the origin and the connection formula$${H^{(1)}_{\nu}}\left(\eta\right)+{H^{(2)}_{\nu}}
\left(\eta\right)=J_{\nu}(\eta).$$
Note that, as $\eta \to 0$, $J_{\nu}\left(\eta\right)$ has the following behavior
$$J_{\nu}\left(\eta\right)\sim(\tfrac{1}{2}\eta)^{\nu}/\Gamma\left(\nu+1\right).$$

It is easy to see from \eqref{H0-origin} that there is a logarithmic singularity at the origin when $\a = \frac{1}{2}$. It is worthwhile to point out that logarithmic behaviors will also appear at the origin when $\a \pm 1/2 \in \mathbb{N} $. This is due to the following relation
\begin{equation}
H^{(1,2)}_{n}(\eta) = (-1)^{n}\eta^{n}\left(\frac{1}{\eta}\frac{\textrm{d}}{\textrm{d}\eta}
\right)^{n}(H^{(1,2)}_0(\eta)), \quad n\in \mathbb{N};
\end{equation}
see \cite[Eq. (10.6.6)]{Olver:NIST}. As a consequence, we conclude that $M$ has a logarithmic singularity at the origin for all $\a =n -\frac{1}{2}$ as $n \in \mathbb{N}$.

\noindent \emph{- For $\eta \in \tilde{\Omega}_{3,4}$}, the calculations are similar.

\item[(ii)] When $0<\a<1/2$: now $\a+1/2>0$ but $\a -1/2<0$.
By the following formulas: see \cite[Eq. 10.4.6]{Olver:NIST};
$$H^{(1)}_{-\nu}(\eta)=e^{\nu\pi i}{H^{(1)}_{\nu}}\left(\eta\right) \quad \textrm{and} \quad H^{(2)}_{-\nu}(\eta)=e^{-\nu\pi i}{H^{(2)}_{\nu}}\left(\eta\right),$$
and the similar computations as in the above case, we obtain the same behaviors of $M$ at the origin.
\end{itemize}
\end{itemize}
As a result, the function $M(\eta)$ defined in \eqref{M-expression}-\eqref{M2} solves the model RH problem for any $\a$.

This finishes the proof of the proposition.
\end{proof}

\begin{rmk}
Note that, in the neighbourhood of interior algebraic singular points, a similar parametrix was constructed in terms of $H^{(1,\,2)}_{\a -\frac{1}{2}}$ and $H^{(1,\,2)}_{\a+\frac{1}{2}}$ in Vanlessen \cite{Vanlessen2003}, where the RH problem has jumps on 8 rays instead of 4 rays in Figure \ref{contour-M}.
\end{rmk}

There are two motivations for us to introduce the RH problem for $M$. The first one is that the RH problem for $M$ and its solution play an important role to extend the asymptotics \eqref{asy-pos} as $x\to +\infty$ from $\a -\frac{1}{2} \notin \mathbb{Z}$ in \cite{Its2003} to any $\a$. Combining the above proposition, we apply a simplified and rotated RH problem of $M$ to achieve this target in Section \ref{sec:extension}.

Second, when we proceed to the nonlinear steepest descent analysis as $x\to -\infty$, the model RH problem for $M$ will be applied to construct the local parametrix at the origin in Section \ref{sec-paramx-0}.

\section{Extension to $\a - \frac{1}{2} \in \mathbb{Z}$ as $x\to +\infty$}\label{sec:extension}

In this section, we apply the Deift-Zhou nonlinear steepest analysis for the original RH problem for $\Psi_\a(\l)$ as $x\to + \infty$. The analysis is similar to that in Its and Kapaev \cite{Its2003}. The novel part is the new local parametrix construction near the origin, which extends the original results in \cite{Its2003} from $\a -\frac{1}{2} \notin \mathbb{Z}$ to any $\a$.

Since $x $ is positive, we introduce the change of variable $\l(z) = x^{1/2}z$ and let $t=x^{3/2}$. Then $\theta(\l)$ in \eqref{theta-def} is transformed into $t\hat{\theta}(z)$ with
$$\hat{\theta}(z):=i(\frac{4}{3}z^3 + z).$$
By taking the normalization $$\widehat{U}(z)=\Psi_{\a}(\l(z))\exp(t\hat{\theta}(z)\sigma_3),$$ the original RH problem for $\Psi_{\a}$ is transformed into the RH problem for $\widehat{U}$ with $\widehat{U}(z)\to I$ as $z\to \infty$ and the jumps $S_k$ in \eqref{Psi-jumps} turn into $e^{-t\hat{\theta}(z) \sigma_3}  S_j e^{t\hat{\theta}(z) \sigma_3} $. This transformation doesn't change the diagonal entries, but multiplies the upper and lower triangular entries of $S_j$ by $e^{\mp 2t \hat{\theta} (z)}$, respectively. Now, one important thing is to check the properties of $\re \hat{\theta} (z) $ in the complex-$z$ plane. It is easy to see that $\hat{\theta}(z)$ has two stationary points at $z_{\pm}= \pm \frac{i}{2}$ and the property of the signature of $\re \hat{\theta}(z)$ is shown in Figure \ref{sign-hat-theta}.
 \begin{figure}[h]
  \centering
  \includegraphics[width=7cm]{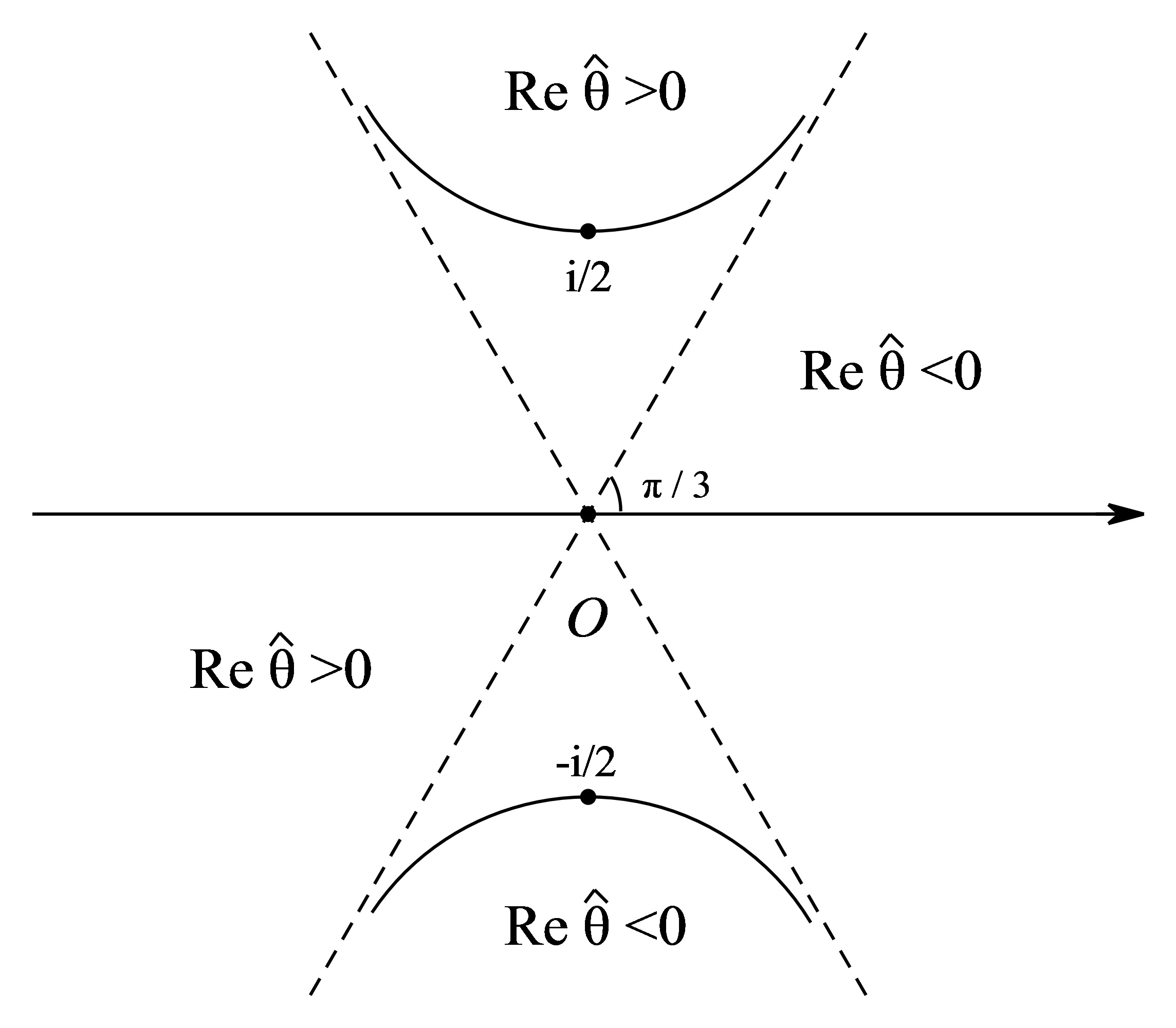}
  \caption{The signature properties of $\re \hat{\theta}(z)$, where the dashed lines are the rays  $ \{z \in \mathbb{C}: \textrm{arg}\, z = \frac{k\pi}{3}, k = 1,2,4,5\}$.}\label{sign-hat-theta}
\end{figure}
\begin{figure}[h]
  \centering
  \includegraphics[width=5cm]{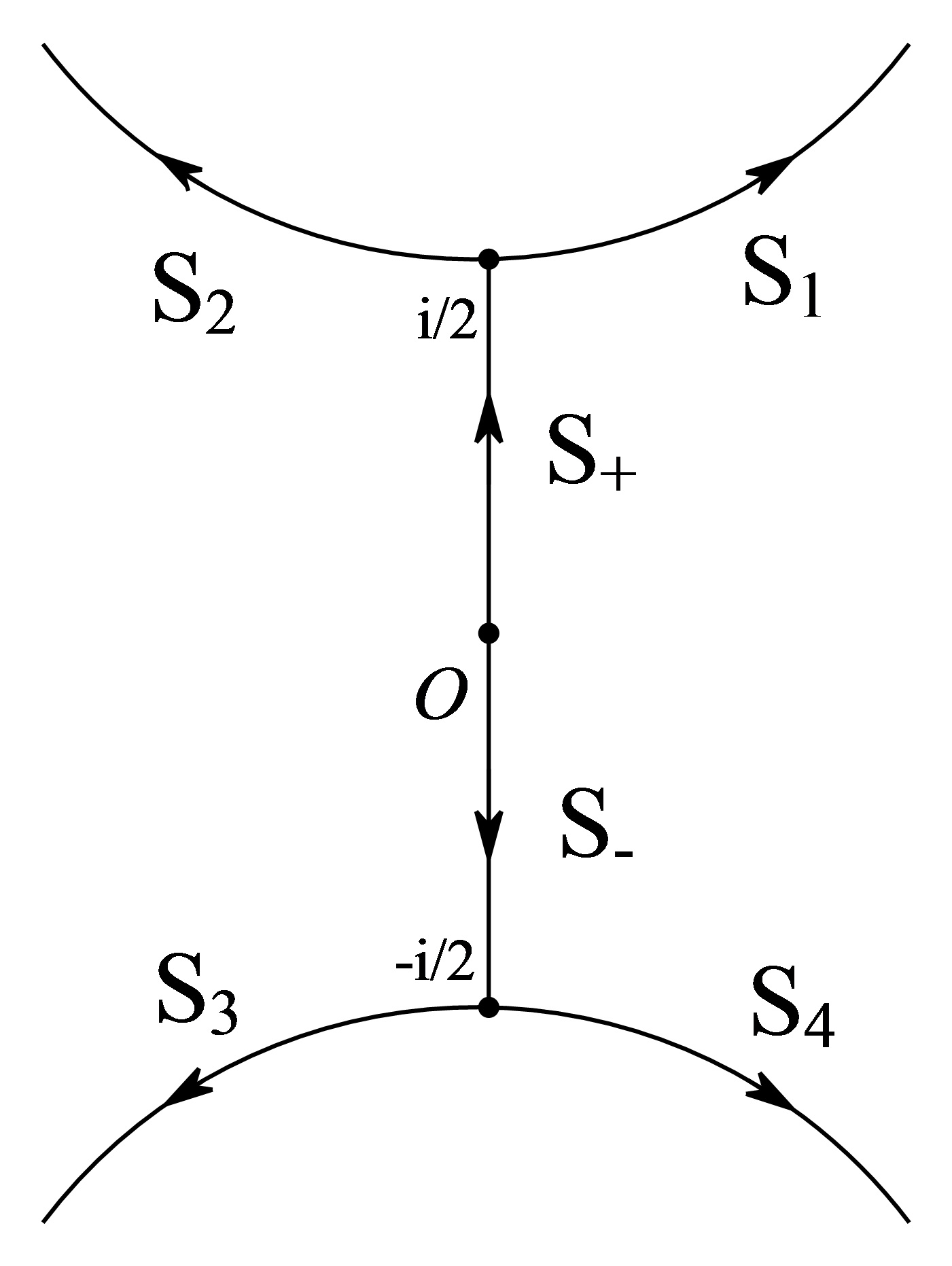}
  \caption{The contour $\Sigma_{\widehat{U}}$ and corresponding jump matrices by ignoring the terms $e^{\pm 2t\hat{\theta}}$.}\label{contour-hat(U)}
\end{figure}
Moreover, we recall that in each sector bounded by the rays $\gamma_j$ in Figure \ref{contour-Arno}, $\Psi_\a(\l)$ is indeed an analytic function with only a branch point at $\l = 0$. Since $z$ is just a rescaling of the variable $\l$, we can deform the original contour $\Sigma$ such that the new one is in accordance with the signature table of $\re \hat{\theta}(z)$ in Figure \ref{sign-hat-theta}. As a result, the original RH problem is transformed into the RH problem for $\widehat{U}$ with the jump conditions shown in Figure \ref{contour-hat(U)}.
\subsection{Asymptotics as $s_3=0$}\label{sec:asy-pos-s3=0}
To derive the asymptotics of $\widehat{U}$ as $x\to +\infty$ (i.e. $t \to +\infty$), we set
\begin{equation}\label{stokes-s3=0}
s_3 = 0,\qquad s_1=-2\sin (\pi \a),
 \end{equation}
 at this moment, which is also used in \cite{Its2003}. Then, the RH problem for $\widehat{U}$ is simplified to the following RH problem for $\widetilde{U}$:
\begin{itemize}
\item[(a)] $\widetilde{U}(z)$ is analytic when $z \in \mathbb{C} \cut l_{\pm}$ with $l_+$ and $l_-$ given in Figure \ref{contour-tildeU-s3=0}.
\begin{figure}[h]
  \centering
  \includegraphics[width=5cm]{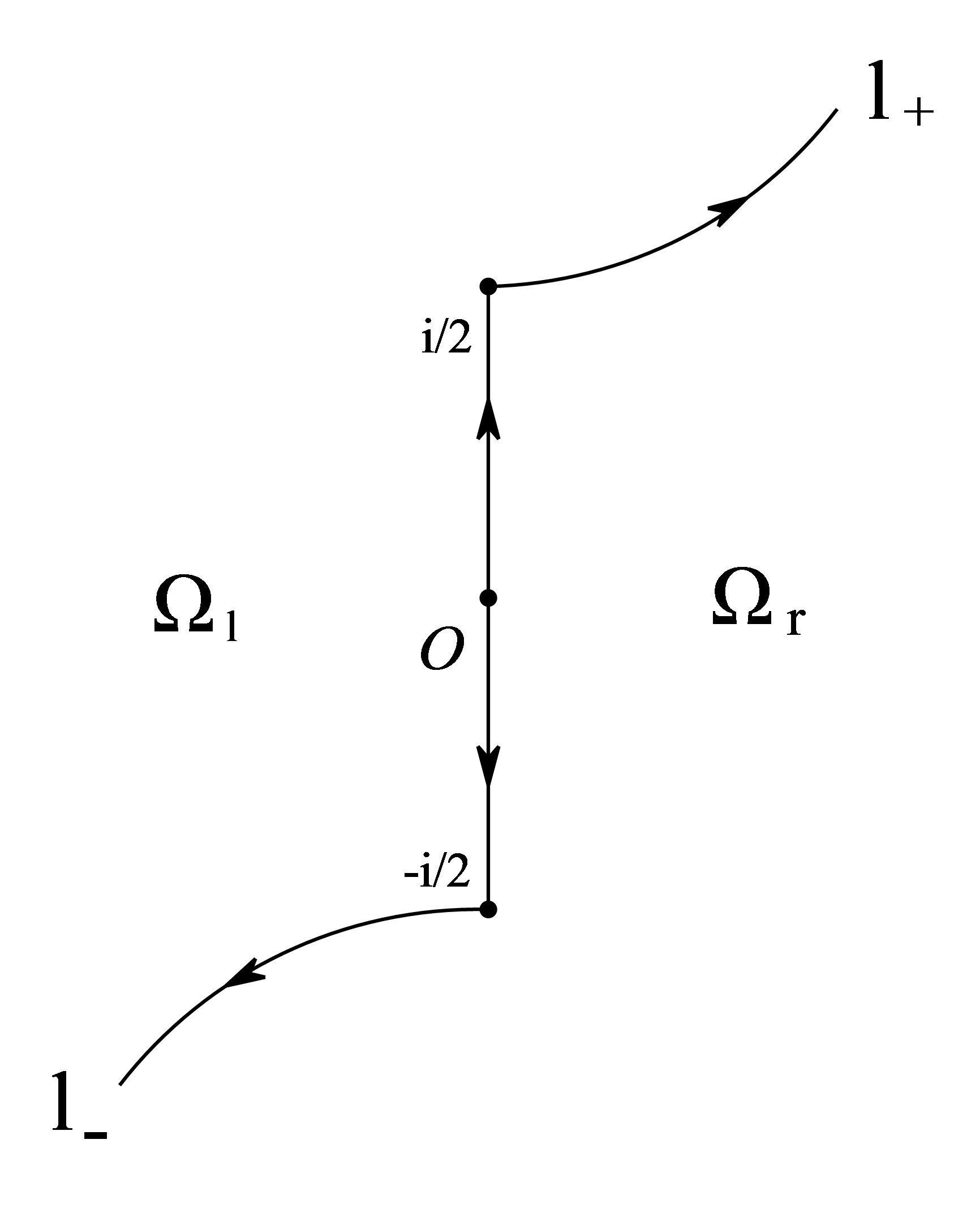}
  \caption{The contour $\Sigma_{\widetilde{U}}$ when $s_3=0$.}\label{contour-tildeU-s3=0}
\end{figure}
\item[(b)] $\widetilde{U}$ satisfies the following jump relations:
\begin{equation} \label{hatU+hatU-1-s3=0}
\widetilde{U}_+(z)= \widetilde{U}_- (z) \left(\begin{matrix}
    1 & 0\\
    -2 \sin (\pi \a) e^{2t\hat{\theta}}& 1
\end{matrix}\right), \qquad \textrm{for } z \in l_+
\end{equation}
and
\begin{equation} \label{hatU+hatU-2-s3=0}
\widetilde{U}_+(z)= \widetilde{U}_- (z) \left(\begin{matrix}
    1 & 2 \sin (\pi \a) e^{-2t\hat{\theta}}\\
    0 & 1
\end{matrix}\right), \qquad \textrm{for } z \in l_-.
\end{equation}
 \item[(c)] As $z \to \infty$, $\widetilde{U}(z)\to I$.
 \item[(d)] At $z =0$, $\widetilde{U}(z)$ is singular in the following form:
\begin{eqnarray}
  \widetilde{U}(z)\left(\begin{matrix}
  1 & i e^{-\pi i \a}\\
   0 & 1
  \end{matrix}\right) =O\left(\begin{matrix}
  |z|^{-\a} &  |z|^{\a}\\
   |z|^{-\a} & |z|^{\a}
  \end{matrix}\right), \quad z \in {\Omega}_{l},\label{asy-hatu-origin-1-s3=0}\\
 \widetilde{U}(z)\left(\begin{matrix}
  1 & 0\\
  -i e^{-\pi i \a} & 1
  \end{matrix}\right) =O\left(\begin{matrix}
  |z|^{\a} &  |z|^{-\a}\\
   |z|^{\a} & |z|^{-\a}
  \end{matrix}\right), \quad z \in \Omega_{r}.\label{asy-hatu-origin-2-s3=0}
\end{eqnarray}
Note that the branch of $z^{\a}$ is chosen arbitrary.
\item[(e)] $\widetilde{U}(z)$ is bounded at $z=\pm \frac{i}{2}$.
\end{itemize}

Combining the property of the signature of $\re \hat{\theta}$ in Figure \ref{sign-hat-theta}, we know that the jump matrices in \eqref{hatU+hatU-1-s3=0} and \eqref{hatU+hatU-2-s3=0} tend to $I$ exponentially fast as $t\to +\infty$ for $z$ bounded away from the origin. As a consequence, $\widetilde{U}(z)$ can be approximated by the identity matrix for $z$ bounded away from the origin. However, because $\hat{\theta}(0)=0$, this no longer holds in the neighborhood of $z=0$. Therefore, we need to construct a parametrix in its neighbourhood.

\subsubsection*{Local parametrix at the origin}
Next, we make use of $M(z)$ in proposition \ref{prop-M} to construct the local parametrix. We consider the following RH problem for $\widetilde{M}$ (a rotation of the RH problem for $M$ when $s_3=0$):

\begin{itemize}
\item[(a)] $\widetilde{M}$ is analytic when $\eta \in \mathbb{C} \cut \Sigma_{\widetilde{M}}$, here $\Sigma_{\widetilde{M}} =\Gamma_{+} \cup \Gamma_{-}$ with $\Gamma_+=\{\eta \in \mathbb{C}: \textrm{arg}\, \eta = -\frac{11\pi}{6}\}$ and $\Gamma_-=\{\eta \in \mathbb{C}: \textrm{arg}\, \eta = -\frac{5\pi}{6}\}$ are two rays oriented to infinity; see Figure \ref{contour(tildeM)}.

    \begin{figure}[h]
  \centering
  \includegraphics[width=8cm]{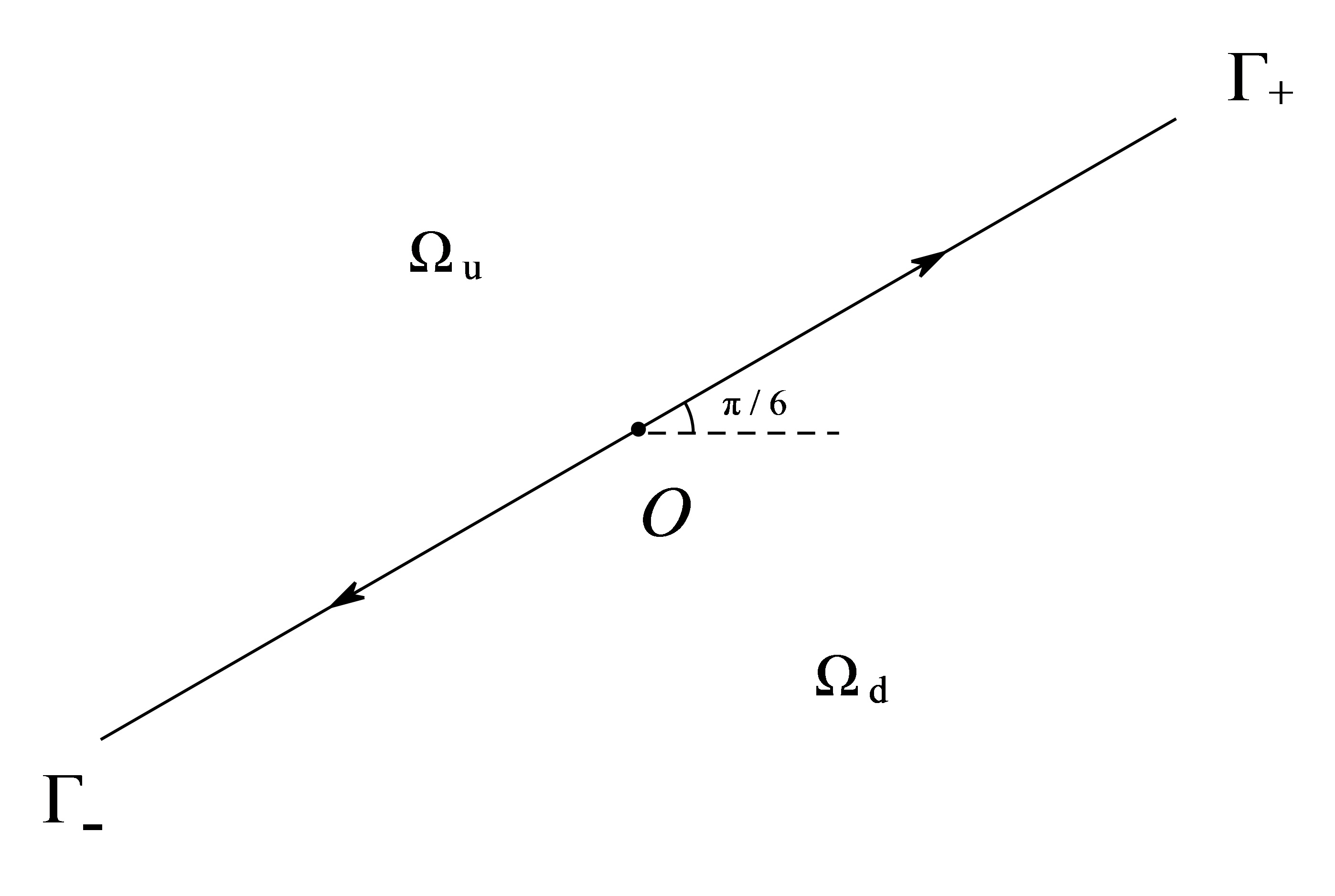}
  \caption{The contour $\Sigma_{\widetilde{M}}$.}\label{contour(tildeM)}
\end{figure}

\item[(b)] $\widetilde{M}$ satisfies the following jump relations:
\begin{equation} \label{tilde-M+M--1}
\widetilde{M}_{+}(\eta)= \widetilde{M}_{-}(\eta) \left(\begin{matrix}
    1 & 2 \sin (\pi \a)\\
    0 & 1
\end{matrix}\right), \qquad \textrm{for } \eta \in \Gamma_{+}
\end{equation}
and
\begin{equation} \label{tilde-M+M--2}
\widetilde{M}_{+}(\eta)= \widetilde{M}_{-}(\eta) \left(\begin{matrix}
    1 & 0\\
    -2 \sin (\pi \a) & 1
\end{matrix}\right), \qquad \textrm{for } \eta \in \Gamma_{-}.
\end{equation}
 \item[(c)] As $\eta \to \infty$, $\widetilde{M}(\eta)$ satisfies the asymptotics:
 \begin{equation}\label{tilde-M-infty-asy}
  \widetilde{M}(\eta)  = (I+O(\eta^{-1})) e^{i\eta\sigma_3}. 
 \end{equation}

 \item[(d)] At $\eta =0$, $\widetilde{M}(\eta)$ is singular in the following form:
\begin{eqnarray}
  \widetilde{M}(\eta)\left(\begin{matrix}
  1 & 0\\
  -i e^{-\pi i \a} & 1
  \end{matrix}\right) =O\left(\begin{matrix}
  |\eta|^{\a} &  |\eta|^{-\a}\\
   |\eta|^{\a} & |\eta|^{-\a}
  \end{matrix}\right), \quad \eta \in \Omega_{u},\label{tilde-M-origin-1}\\
  \widetilde{M}(\eta)\left(\begin{matrix}
  1 & i e^{-\pi i \a}\\
   0 & 1
  \end{matrix}\right) =O\left(\begin{matrix}
  |\eta|^{-\a} &  |\eta|^{\a}\\
   |\eta|^{-\a} & |\eta|^{\a}
  \end{matrix}\right), \quad \eta \in \Omega_{d},\label{tilde-M-origin-2}
\end{eqnarray}
where the branch cuts of $\eta^{\pm\a}$ are chosen along $\Gamma_+$.
\end{itemize}

From the construction of solution to RH problem for $M$ and the definition of $M_{1,2}$ in \eqref{M1}-\eqref{M2}, we know that
\begin{equation}\label{tilde-M-def}
\widetilde{M}(\eta)=\begin{cases}\D
M_2 (\eta),& \quad \eta \in \Omega_u\\
M_1 (\eta), & \quad \eta \in \Omega_d
\end{cases}
\end{equation}
solves the RH problem for $\widetilde{M}$.

Next, we introduce the conformal mapping in $U(0,\delta)$ with $\delta< 1/2$ a small positive constant:
\begin{equation}\label{eta-z-pos}
\eta(z)=-\frac{4}{3}z^{3}- z.
\end{equation}
It is easy to check that $\widetilde{M}(t\eta(z))e^{-it\eta(z)\sigma_3}$ solves the RH problem for $\widetilde{U}$ in the neighborhood of origin $U(0, \delta)$.

 \begin{rmk}
   Recall that the local parametrix near the origin in \cite{Its2003} is constructed in terms of Bessel functions of the first kind $J_{\a -\frac{1}{2}}$ and $J_{\frac{1}{2}-\a}$,  which are \emph{linearly dependent} when $\a-\frac{1}{2} \in \mathbb{Z}$. This is also the reason why  their asymptotic results do not hold for $\a-\frac{1}{2} \in \mathbb{Z}$. However, the functions $H^{(1,\,2)}_{\a -\frac{1}{2}}$ and $H^{(1,\,2)}_{\a+\frac{1}{2}}$ are linearly independent for any $\a \in \mathbb{R}$. Therefore, with the function $\widetilde{M}$ defined in \eqref{M1}-\eqref{M2} and \eqref{tilde-M-def}, we extend the asymptotics \eqref{u1} to any $\a \geq 0$ successfully.
 \end{rmk}

\subsubsection*{Error estimation}
Next, we set
\begin{equation}
\widetilde{T}(z)=\begin{cases}\D
\widetilde{M}(t\eta(z))e^{-it\eta(z)\sigma_3},& \quad |z|<\delta\\
I, & \quad |z|>\delta,
\end{cases}
\end{equation}
and consider the error function
\begin{equation}
\widetilde{R}(z)=\widetilde{U}(z) \widetilde{T}^{-1}(z).
\end{equation}
It is easy to deduce that $\widetilde{R}$ solves the following RH problem:
\begin{itemize}

\item[(a)] $\widetilde{R}_{+}(z)=\widetilde{R}_{-}(z) J_{\widetilde{R}}(z)$
with
\begin{equation}\label{jump-tilde-R}
J_{\widetilde{R}}(z)=\begin{cases}\D
\left(\begin{matrix}
    1 & 0\\
    -2 \sin (\pi \a)e^{2t\hat{\theta}} & 1
\end{matrix}\right), & \quad z \in \gamma_+,\,\, |z|>\delta, \\
\left(\begin{matrix}
    1 & 2 \sin (\pi \a) e^{-2t\hat{\theta}}\\
    0  & 1
\end{matrix}\right), & \quad z \in \gamma_-, \,\,|z|>\delta, \\
\widetilde{M}(t\eta(z))e^{-it\eta(z) \sigma_3}, & \quad |z|=\delta.
\end{cases}
\end{equation}

\item[(b)] $\widetilde{R}(z) \to I$ when $z \to \infty$.

\end{itemize}

The above jump matrix $J_{\widetilde{R}}(z)$ satisfies the following estimates:
\begin{equation}
||J_{\widetilde{R}}(z)-I||\leq \begin{cases}\D
c_1 e^{-c_2 t}, & \quad z \in \gamma_{\pm}, \,\, |z|>\delta,\\
c_1 t^{-1}, & \quad|z|=\delta,
\end{cases}
\end{equation}
where $c_{1,2,3}$ are certain positive constants. Then, the RH problem for $\widetilde{R}$ is solvable in terms of the following Cauchy integral
\begin{equation}\label{tilde-R-Cau-int}
\widetilde{R}_{-}(z)=I +\frac{1}{2\pi i}\int_{\Sigma_{\widetilde{R}}}\widetilde{R}_{-}(z')(J_{\widetilde{R}}(z')-I)\frac{\textrm{d} z'}{z' -z_{-}}, \quad z \in \Sigma_{\widetilde{R}}.
\end{equation}
Based on the standard procedure of norm estimation of the above Cauchy operator, it follows that, for sufficiently large $t$, the relevant integral operator is contracting, and the integral equation can be solved in $L^2(\Sigma_{\widetilde{R}})$ by iterations; see the standard arguments in \cite{Deift1999,Dei:Kri:McL:Ven:Zhou1999-2}.
Then, we have, as $t \to +\infty$,
\begin{equation}\label{est-tilde-R}
\widetilde{R}(z) =I + O(t^{-1}), \qquad \textrm{uniformly for } z\in \mathbb{C} \setminus \Sigma_{\widetilde{R}}.
\end{equation}

\subsubsection*{Asymptotics as $ x\to +\infty$ for $s_3=0$}
Based on the above analysis and \eqref{PII-RHP-rel}, we have
\begin{equation}
u_1(x;\a)=2\sqrt{x}\lim_{z\to \infty}z \widetilde{R}_{12}=-\frac{\sqrt{x}}{\pi i} \int_{\Sigma_{\widetilde{R}}} \biggl(\widetilde{R}_{-}(z')\left(J_{\widetilde{R}}(z')-I\right) \biggr)_{12} \textrm{d}z'.
\end{equation}
Combining \eqref{est-tilde-R}, the above equation gives the following formula
\begin{equation}
u_1(x;\a)=-\frac{\sqrt{x}}{\pi i} \int_{\Sigma_{\widetilde{R}}} \biggl(\left(J_{\widetilde{R}}(z)-I\right) \biggr)_{12} \textrm{d}z + O(x^{-5/2}).
\end{equation}
Using the definition of the jump matrix $H(z)$ in \eqref{jump-tilde-R}, \eqref{tilde-M-infty-asy} and \eqref{eta-z-pos}, we find from the above formula that
\begin{equation}
u_1(x;\a)\sim \frac{\a}{x}+O(x^{-5/2}), \quad x\to +\infty.
\end{equation}

Moreover, let $K$ denote the operator such that
\begin{equation}\label{operator-K}
K[f(z)]:=\frac{1}{2\pi i}\int_{\Sigma_{\widetilde{R}}}f(z)(J_{\widetilde{R}}(z')-I)\frac{\textrm{d} z'}{z' -z_{-}}, \quad z \in \Sigma_{\widetilde{R}},
\end{equation}
then \eqref{tilde-R-Cau-int} gives us
\begin{equation}
\widetilde{R}_{-}(z)-I = K I + K [\widetilde{R}_{-}(z)-I], \quad z \in \Sigma_{\widetilde{R}}.
\end{equation}
Since $||K||_{L_2}\leq c x^{-3/2}$ as $x\to +\infty$, we have
 $\widetilde{R}_{-} =\sum_{n=0}^{\infty}K^{n}I$, which is a converging iterative series and suggests $u_{1}(x;\a)$ possesses an asymptotic series in terms of negative degrees of $x^{1/2}$ when $x\to +\infty$. Then, the meromorphicity of $u_1(x;\a)$ leads to the following asymptotic expansion:
\begin{equation}\label{u1}
    u_1(x;\a)\sim \frac{\a}{x}\sum_{n=0}^{\infty}a_n x^{-3n},\quad x\to +\infty,
\end{equation}
with $a_n$ determined by the recurrence relation \eqref{a_n-rec-rel}.

\subsection{Asymptotics for $s_3\neq 0$}\label{sec:asy-pos}

Note that $u_1(x;\a)$ is not the solution studied in Theorem \ref{main-thm}, because the values of the stokes multiplies are different in \eqref{stokes-real-as} and \eqref{stokes-s3=0}. To get the solution in Theorem \ref{main-thm}, we now consider the RH problem for $\widehat{U}$ when $s_3\neq 0$ and look for the solution in terms of the following form:
\begin{equation}
 \widehat{U}(z)=X(z)\widetilde{U}(z).
\end{equation}
It is easy to check that $X(z)$ solves the following RH problem:

\begin{itemize}
\item[(a)] $X(z)$ is analytic in $\mathbb{C}\cut {\Sigma_X}$; see Figure \ref{contour(X)}.
\item[(b)] $X_{+}(z)=X_{-}(z) J_X(z)$
with
\begin{equation}\label{jump-X}
J_X(z)= \begin{cases}\D
\Phi_{\a,-}(z)\left(\begin{matrix}
    1 & 0\\
    -s_3 e^{2t\hat{\theta}} & 1
\end{matrix}\right)\Phi^{-1}_{\a,+}(z), &\quad z \in \gamma_+,\\
\Phi_{\a,-}(z)\left(\begin{matrix}
    1 & s_3 e^{-2t\hat{\theta}}\\
    0  & 1
\end{matrix}\right)\Phi^{-1}_{\a,+}(z), &\quad z \in \gamma_-.
\end{cases}
\end{equation}

   \begin{figure}[h]
  \centering
  \includegraphics[width=5cm]{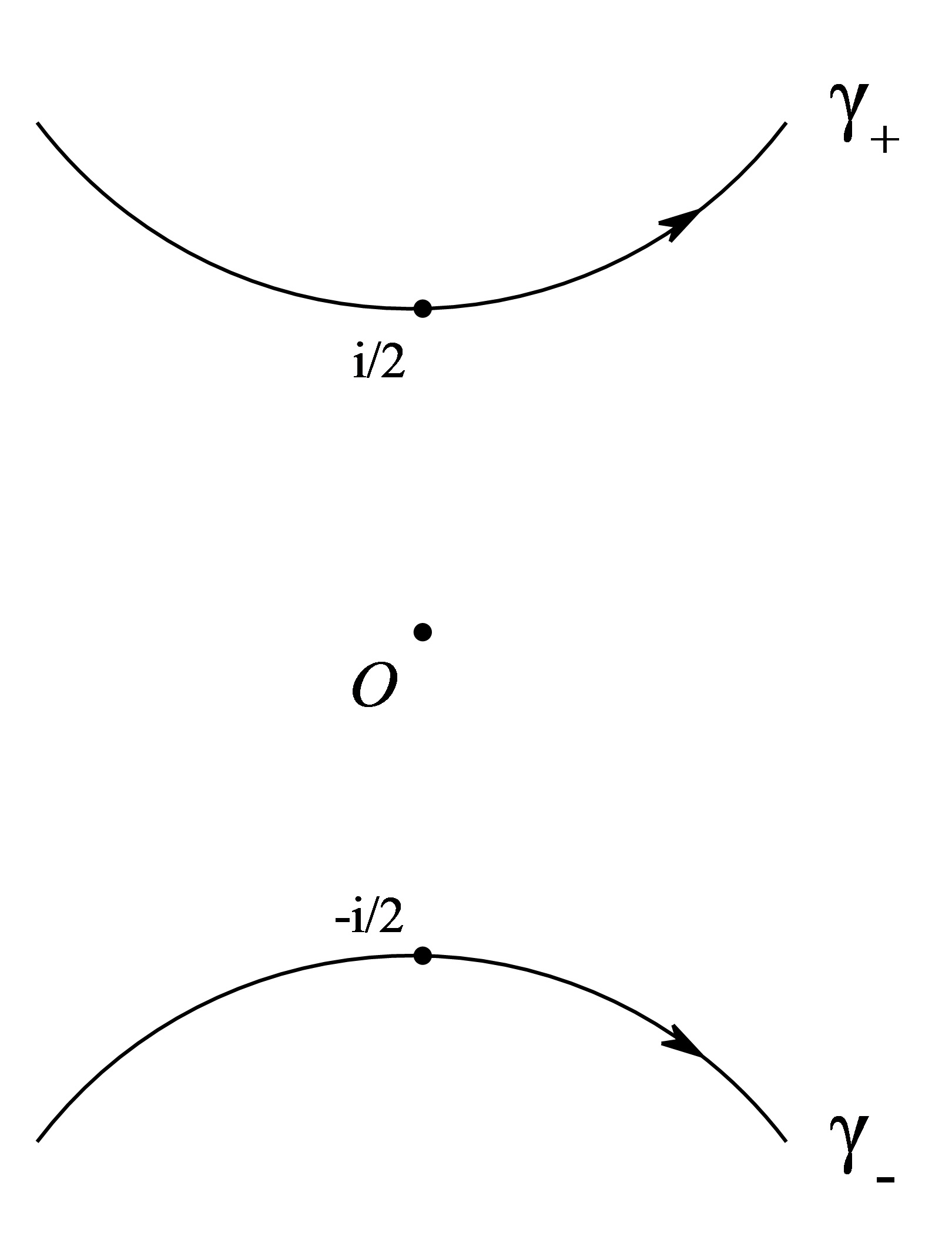}
  \caption{The contour $\Sigma_X$.}\label{contour(X)}
\end{figure}
\item[(c)] $X(z) \to I$ when $z \to \infty$.
\item[(d)] $X(z)$ is bounded at the origin.
\end{itemize}

It is interesting to note that $X(z)$ and the function $X(\l)$ in \cite[P. 378]{Its2003} are simply related through the following transformation
\begin{equation}
X(z) = X(\l(z))\exp(t\hat{\theta}\sigma_3).
\end{equation}
Then, following similar analysis as that used to derive  \cite[Thm. 2.2]{Its2003}, we have the following theorem for $\a\geq 0$:
\begin{theorem}\label{thm-extension}
There exists a family of solutions of the PII equation satisfying the following asymptotics:
\begin{equation}\label{u-u1}
u(x;\a)= u_1(x;\a)-\frac{is_3}{2\sqrt{\pi}}x^{-1/4}(1+O(x^{-3/4})) \quad \textrm{as\,\,} x\to +\infty,
\end{equation}
where $\a\in \mathbb{R}$ and $u_1(x;\a)\sim \a/x$ is the solution of the PII equation shown in \eqref{u1}.
\end{theorem}

For $s_1=-\sin(\pi \a) - ki$, $s_3=-\sin(\pi \a) + ki$ and $k, \a \in \mathbb{R}$, we know from Remark \ref{rmk-real-reduction} that $u(x;\a)$ is real but $u_1(x;\a)$ is not real for real $x$. Thus, one needs to take the real part of \eqref{u-u1} to get the asymptotics of $u(x;\a)$. As a result, we obtain the asymptotic behavior in \eqref{asy-pos}.

\section{Nonlinear steepest descent analysis as $x\to -\infty$}\label{sec-non-anal}

 The analysis in the first several steps is literally the same as that in \cite{Dai:Hu2017}, where we studied the asymptotics for the AS solutions of PII under the conditions $s_1=\bar{s}_3$ and $1-s_1s_3>0$. Here we repeat some of the arguments to make the paper self-contained. We first take the change of variable $\l(z)=(-x)^{1/2}z$, so $\theta(\l) $ in \eqref{Psi-infty-asy} is rewritten by $t\tilde{\theta}(z)$ with
 \begin{equation}
   \tilde{\theta}(z) := i(\frac{4}{3}z^3 -z) \quad \textrm{and} \quad  t=(-x)^{3/2}.
 \end{equation}
 Then, under the normalization $$U(z) = \Psi_\a(\l(z))\exp(t \tilde{\theta}(z)\sigma_3),$$ we have $U(z)\to I$ as $z\to \infty$. Similar to the case when $x\to +\infty$, the jumps $S_k$ in \eqref{Psi-jumps} turn into $e^{-t\tilde{\theta}(z) \sigma_3}  S_j e^{t\tilde{\theta}(z) \sigma_3} $. Also, this transformation doesn't change the diagonal entries, but multiplies the upper and lower triangular entries of $S_j$ by $e^{\mp 2t \tilde\theta (z)}$, respectively. Next, we check the properties of $\re \tilde\theta (z) $ in the complex-$z$ plane. It is easy to see that $\tilde{\theta}(z)$ has two stationary points at $z_{\pm}= \pm \frac{1}{2}$ and the property of the signature of $\re \tilde{\theta}(z)$ is shown in Figure \ref{sign-theta}.
\begin{figure}[h]
  \centering
  \includegraphics[width=10cm]{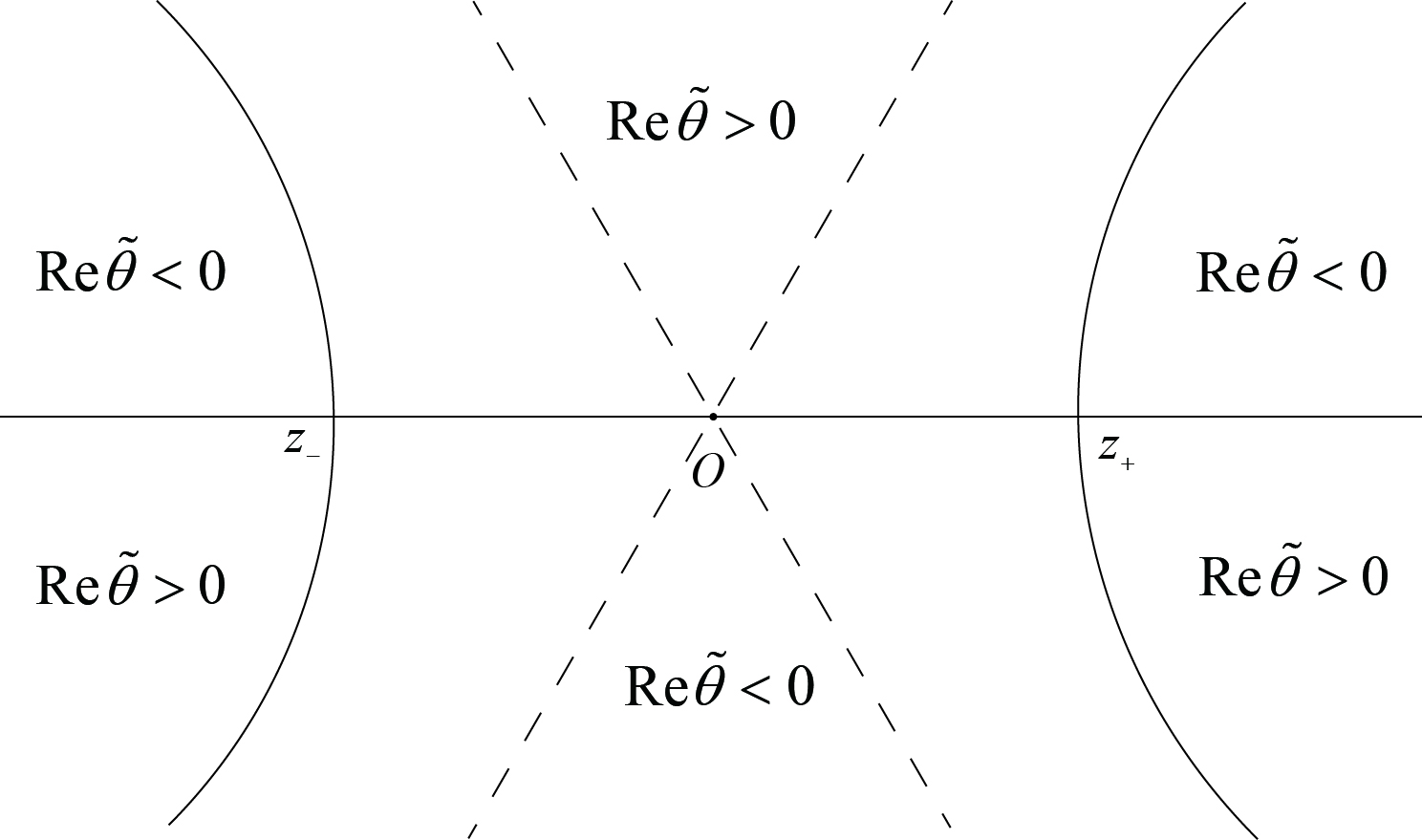}
  \caption{The signature properties of $\re \tilde{\theta}(z)$, where the dashed lines are the rays  $ \{z \in \mathbb{C}: \textrm{arg}\, z = \frac{k\pi}{3}, k = 1,2,4,5\}$. }\label{sign-theta}
\end{figure}
Moreover, we recall that in each sector bounded by the rays $\gamma_j$ in Figure \ref{contour-Arno}, $\Psi_\a(\l)$ is indeed an analytic function with only a branch point at $\l = 0$. Since $z$ is just a rescaling of the variable $\l$, we can deform the original contour $\Sigma$ such that the new one
is in accordance with the signature table of $\re \tilde{\theta}(z)$ in Figure \ref{sign-theta}. As a result, the original RH problem is transformed into the RH problem for $U$ with the jump conditions shown in Figure \ref{contour(sd)}.
\begin{figure}[h]
  \centering
  \includegraphics[width=14cm]{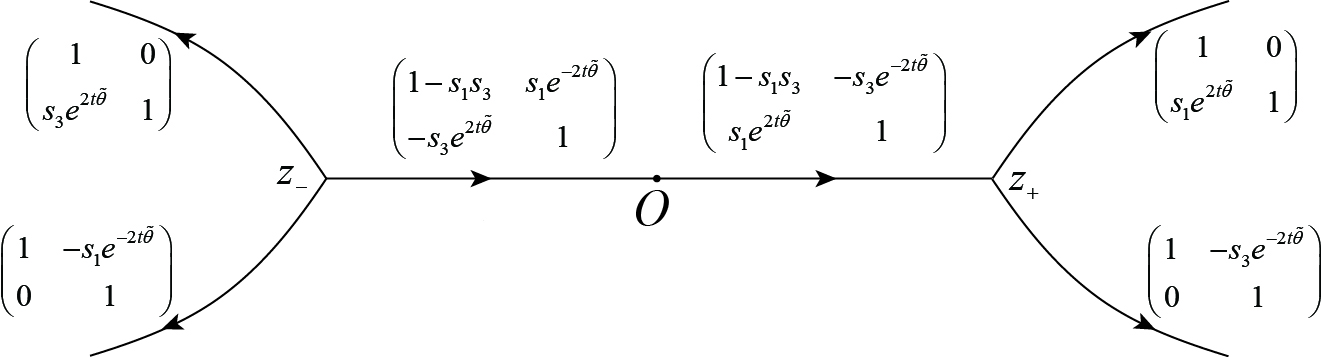}
  \caption{The steepest descent contour $\Sigma_U$ with corresponding jump matrices.}\label{contour(sd)}
\end{figure}
\begin{figure}[h]
  \centering
  \includegraphics[width=14cm]{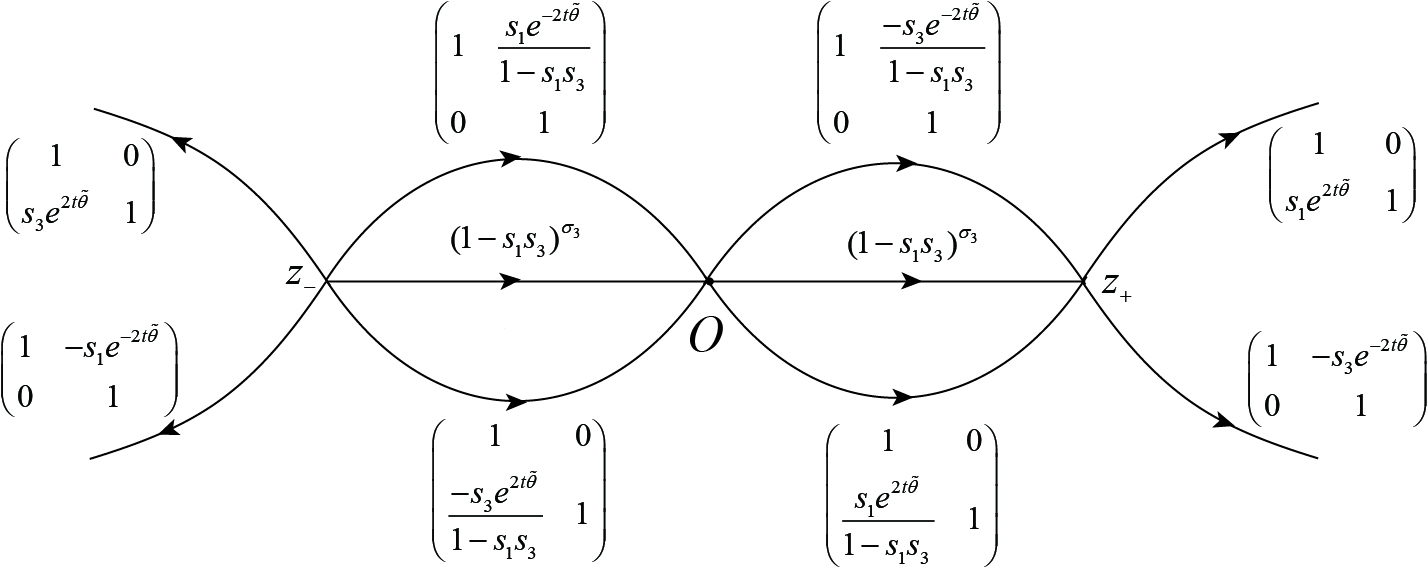}
  \caption{The contour $\Sigma_T$ with corresponding jump matrices.}\label{contour(LD)}
\end{figure}

Secondly, on $[z_-, z_+]$, we factorize the above two matrices by LDU-decomposition:
\begin{eqnarray}
\left(
  \begin{array}{cc}
  1-s_1 s_3 & -s_3 e^{-2t \tilde \theta}\\
    s_1 e^{2t \tilde \theta} & 1
  \end{array}
    \right) &=&  \left(
  \begin{array}{cc}
  1 & 0\\
    \frac{s_1 e^{2t\tilde{\theta}(z)}}{1-s_1 s_3} & 1
  \end{array}
    \right)  \left(
  \begin{array}{cc}
  1-s_1 s_3 & 0\\
    0 & \frac{1}{1-s_1 s_3}
  \end{array}
    \right) \left(
  \begin{array}{cc}
  1 & -\frac{s_3 e^{-2t\tilde{\theta}(z)}}{1-s_1 s_3}\\
    0 & 1
  \end{array}
    \right) \nonumber\\
  &:=&S_{L_1}S_{D}S_{U_1},
\end{eqnarray}
and
\begin{eqnarray}
\left(
  \begin{array}{cc}
  1-s_1 s_3 & s_1 e^{-2t \tilde \theta}\\
    -s_3 e^{2t \tilde \theta} & 1
  \end{array}
    \right)&=& \left(
  \begin{array}{cc}
  1 & 0\\
    -\frac{s_3 e^{2t\tilde{\theta}(z)}}{1-s_1 s_3} & 1
  \end{array}
    \right)  \left(
  \begin{array}{cc}
  1-s_1 s_3 & 0\\
    0 & \frac{1}{1-s_1 s_3}
  \end{array}
    \right) \left(
  \begin{array}{cc}
  1 & \frac{s_1 e^{-2t\tilde{\theta}(z)}}{1-s_1 s_3}\\
    0 & 1
  \end{array}
    \right)\nonumber \\
    &:=&S_{L_2}S_{D}S_{U_2}.
\end{eqnarray}
From the above factorization, we perform an \lq\lq opening of lenses \rq\rq; see Figure \ref{contour(LD)}. Then we introduce the function
 \begin{equation}\label{def-T}
T(z) := \begin{cases}\D
U(z), & \textrm{for $z$ outside the two lens shaped regions},\\
U(z)S^{-1}_{U_1}, & \textrm{for $z$ in the upper part of the right lens shaped region},\\
U(z)S_{L_1}, & \textrm{for $z$ in the lower part of the right lens shaped region},\\
U(z)S^{-1}_{U_2}, & \textrm{for $z$ in the upper part of the left lens shaped region},\\
U(z)S_{L_2}, & \textrm{for $z$ in the lower part of the left lens shaped region}.
\end{cases}
\end{equation}
It is easy to check that $T$ satisfies the RH problem shown in \ref{contour(LD)}.

Note that, according to the signature of $\re \tilde\theta (z)$, we know that the jump matrices $J_{T}$ tend to the identity matrix exponentially when $t\to +\infty$ as long as $z$ stays away from the segment $[z_{-}, z_{+}]$. As a result, the key contribution to the asymptotic expansion of $T$ comes from $[z_{-}, z_{+}]$ and the neighborhoods of $z_{\pm}$.

\subsection{Global parametrix on $[z_{-}, z_{+}]$}

As $t\to +\infty$, consider the segment $[z_{-}, z_{+}]$ with the same orientation shown in Figure \ref{contour(LD)}, we obtain a function $N(z)$ satisfies:

\begin{itemize}
  \item[(a)] $N(z)$ is analytic when $z\in \mathbb{C} \setminus [z_-, z_+]$.

  \item[(b)] On $[z_-, z_+]$,
\begin{align}
N_{+}(z)=N_{-}(z)(1-s_1 s_3)^{\sigma_3}.
\end{align}

\item[(c)]  $N(z)= I+O(\frac{1}{z}) $, \qquad as $z\to \infty$.

\end{itemize}

According to \cite{Dai:Hu2017,Fokas2006}, we solve the above RH problem by defining
\begin{equation} \label{parametrix-global}
  N(z)=\left(\frac{z-z_{-}}{z-z_{+}}\right)^{\nu \sigma_3}, \qquad \textrm{with } \nu= -\frac{1}{2\pi i} \ln(1-s_1 s_3).
\end{equation}
Here, the branch cut is chosen such that $\arg (z- z_{\pm}) \in (-\pi, \pi)$, then $\left(\frac{z-z_{-}}{z-z_{+}}\right)^{\nu} \to 1$ as $ z\to \infty$.

\subsection{Local parametrices near $z = \pm \frac{1}{2}$ } \label{sec:local-z+}

In this subsection, we seek a function $Q^{(r)}(z)$ satisfying the same problem as $T$ in $U(z_{+}, \delta)$, a neighborhood of $z_+=\frac{1}{2}$.
\begin{itemize}
  \item[(a)] $Q^{(r)}(z)$ is analytic when $z\in U(z_{+}, \delta) \setminus \Sigma_T$.

  \item[(b)] When $z \in U(z_{+}, \delta)\cap \Sigma_T$, $Q^{(r)}(z)'s$ jump matrices are the same as that of $T(z)$; see Figure \ref{contour(psi-r)}.

  \begin{figure}[h]
  \centering
  \includegraphics[width=8cm]{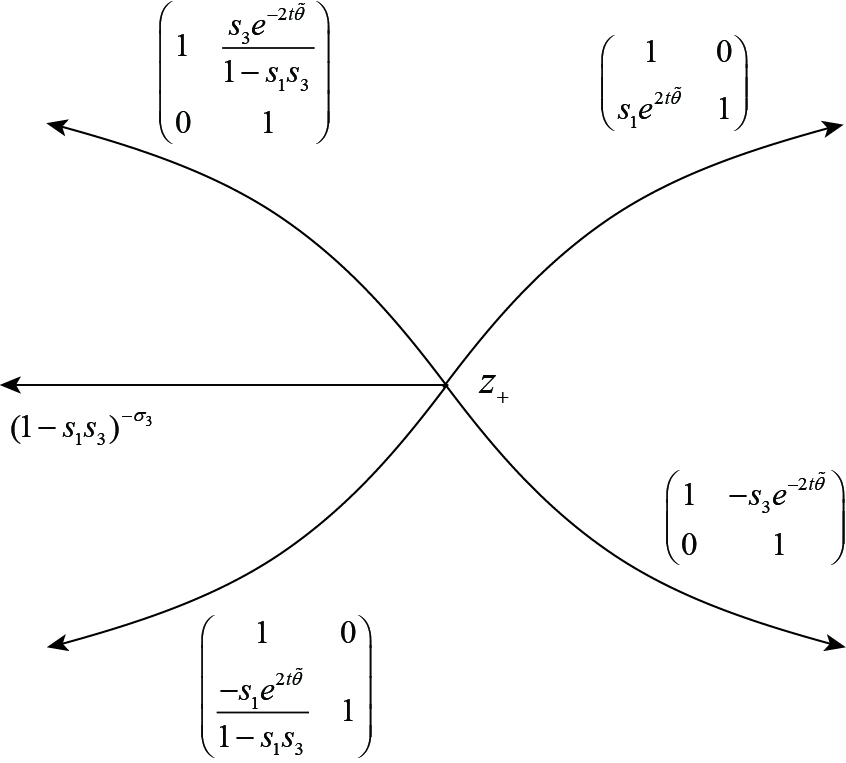}
  \caption{The contour  $\Sigma_T $ in $U(\frac{1}{2}, \delta)$ and corresponding jump matrices.}\label{contour(psi-r)}
\end{figure}

  \item[(c)]On $\partial U(z_{+}, \delta) = \{ z \in \mathbb{C}: \ |z - z_{+}| = \delta \}$,
  \begin{equation} \label{psir-psid-match}
   Q^{(r)}(z) [N(z) ]^{-1} = I + O (1/t), \qquad \textrm{as } t \to + \infty.
  \end{equation}

  \item[(d)] As $z \to z_{+}$, $Q^{(r)}(z)=O((z-z_{+})^{-1})$.
\end{itemize}

The above RH problem of $Q^{(r)}(z)$ is similar to that of $P^{(r)}(z)$ in \cite[Sec. 3.4]{Dai:Hu2017}, except that $P^{(r)}(z)$ is bounded at $z_{+}$ in \cite{Dai:Hu2017}, while $z_{+}$ is a simple pole of $Q^{(r)}(z)$. The difference is caused by the matching condition \eqref{psir-psid-match} and more details can be seen later in the current section.

It is natural to recall the construction of $P^{(r)}(z)$; see \cite[Sec. 3.4]{Dai:Hu2017} and Fokas et al. \cite[Sec. 9.4]{Fokas2006}, since the jumps keep the same by multiplying a function, which has poles, at the left hand side of $P^{(r)}(z)$. For the self-consistency, we will repeat some definitions of functions in the rest of this section.

We define a conformal mapping in $U(z_{+}, \delta)$ in the following form:
\begin{align} \label{zeta-defn}
\zeta(z):=2\sqrt{-\tilde{\theta}(z)+\tilde{\theta}(z_+)} =  2  \, \sqrt{-\frac{4i}{3} z^3+ iz -\frac{i}{3}},
\end{align}
with $\arg{(z-z_{+})} \in (-\pi, \pi)$. Then we have
\begin{equation} \label{zeta-z+-nei}
\zeta(z) \sim e^{3i\pi/4} 2\sqrt{2} \, (z-\frac{1}{2}), \qquad \textrm{as } z \to z_{+},
\end{equation}
with $z_{+}=\frac{1}{2}$.
In \cite{Dai:Hu2017, Fokas2006}, the authors solve $P^{(r)}(z)$ by defining
\begin{equation}\label{Pr-def}
  P^{(r)}(z)= \beta(z)^{\sigma_3} \left(\frac{-h_1}{s_3}\right)^{-\sigma_3 /2} e^{it\sigma_3 /3} 2^{-\sigma_3 /2}\left(\begin{matrix}
\sqrt{t}\,\zeta(z) & 1\\
1 & 0
\end{matrix}\right) Z(\sqrt{t}\,\zeta(z)) e^{t\tilde{\theta}(z)\sigma_3}  \left(\frac{-h_1}{s_3}\right)^{\sigma_3 /2},
\end{equation}
where $h_1$ is a constant with
\begin{equation}\label{h1-def}
 h_1 = \frac{\sqrt{2\pi}}{\Gamma(-\nu)}e^{i \pi \nu}
\end{equation}
and $\beta(z)$ is holomorphic in $U(z_{+}, \delta)$ with
\begin{align} \label{beta-defn}
\beta(z):=\left(\sqrt{t}\,\zeta(z)\frac{z+1/2}{z-1/2}\right)^{\nu }, \quad  \quad  \beta(z_+) = (8t)^{\nu /2}e^{3i\pi \nu /4}.
\end{align}
Recall that $Z(\zeta)$ is a function defined based on the parabolic cylinder functions, the exact definition is as follows:
\begin{equation} \label{zrh-def}
Z(\zeta):= \begin{cases}\D
Z_0(\zeta), &\arg \zeta \in (-\frac{\pi}{4}, 0),\\
Z_k(\zeta), &\arg \zeta \in (\frac{k-1}{2} \pi, \frac{k}{2}\pi), \ k = 1,2,3,\\
Z_4(\zeta), &\arg \zeta \in (\frac{3\pi}{2}, \frac{7\pi}{4}),
\end{cases}
\end{equation}
where
\begin{equation} \label{z0-def}
Z_0(\zeta)= 2^{-\sigma_3/2} \left(
  \begin{matrix}
  D_{-\nu-1}(i\zeta) & D_{\nu}(\zeta)\\
   \frac{d}{d\zeta}D_{-\nu-1}(i\zeta) & \frac{d}{d\zeta}D_{\nu}(\zeta)
  \end{matrix}
    \right) \left(
  \begin{matrix}
  e^{\frac{i\pi}{2}(\nu + 1)} & 0\\
    0 & 1
  \end{matrix}
    \right),
\end{equation}
and
\begin{eqnarray}
Z_{k+1}(\zeta) = Z_k (\zeta)H_k, \qquad k = 0,1,2,3.
\end{eqnarray}
The constant matrices $H_j$ above are defined as follows
\begin{equation} \label{def-h-0123}
  H_0= \left(
  \begin{matrix}
  1 & 0\\
    h_0 & 1
  \end{matrix}
    \right), \quad H_1= \left(
  \begin{matrix}
 1 & h_1\\
    0 & 1
  \end{matrix}
    \right), \quad H_{k+2}=e^{i\pi (\nu+\frac{1}{2})\sigma_3}H_k e^{-i\pi (\nu+\frac{1}{2})\sigma_3}, \quad k = 0,1,
\end{equation}
with
\begin{equation}\label{def-h-01}
  h_0 = -i\frac{\sqrt{2\pi}}{\Gamma(\nu+1)} \qquad \textrm{and} \qquad h_1 = \frac{\sqrt{2\pi}}{\Gamma(-\nu)}e^{i \pi \nu}.
\end{equation}
Note that $Z_k(\zeta)$'s defined above are indeed entire functions of $\zeta$.
Let $J_{Z}$ denotes the jump matrices $H_k$, then $Z(\zeta)$ satisfies the following RH problem:
\begin{itemize}
  \item[(a)] $Z(\zeta)$ is analytic for $ \zeta \in \mathbb{C} \cut \Sigma_Z$;

  \item[(b)] $Z_+(\zeta)=Z_{-}(\zeta) J_{Z}(\zeta)$ for $\zeta \in \Sigma_Z$, where the contour $\Sigma_Z$ and the jump $J_Z(\zeta)$ are depicted in Figure \ref{contour(Z)};

  \begin{figure}[h]
  \centering
  \includegraphics[width=8cm]{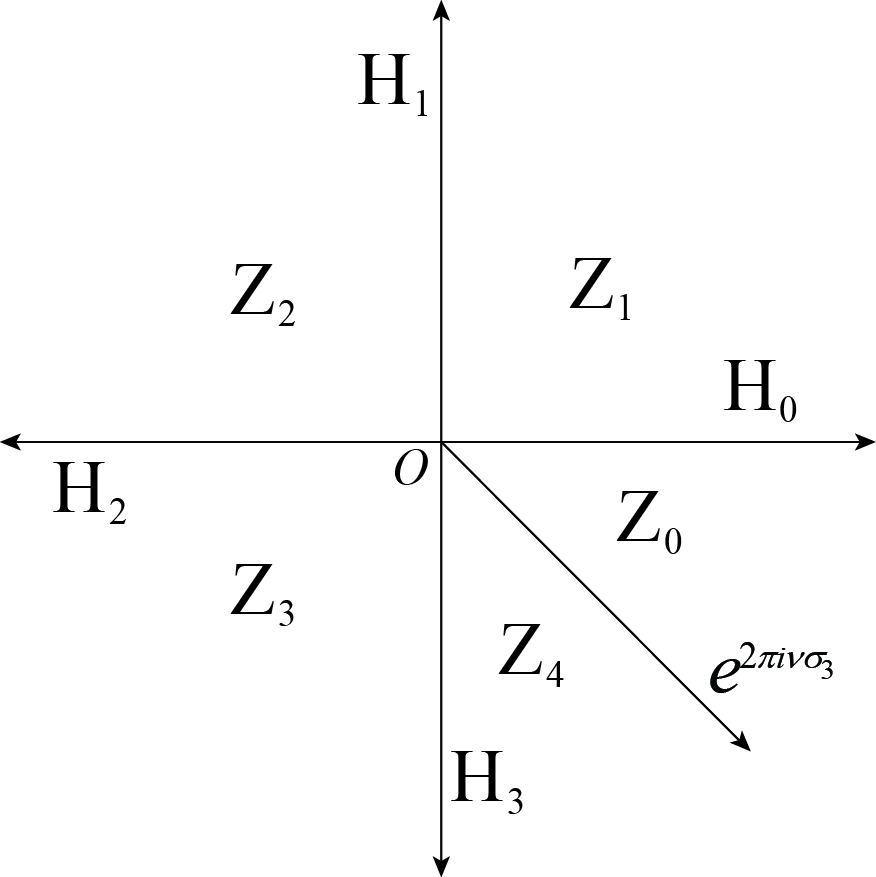}
  \caption{The contour $\Sigma_Z$ and associated jump matrices for the RH problem of $Z(\zeta)$.}\label{contour(Z)}
  \end{figure}

  \item[(c)] As $\zeta \to 0$, $Z(\zeta)$ is bounded;

  \item[(d)] As $\zeta \to \infty$, we have
    \begin{equation} \label{Z-large}
    Z(\zeta)= \zeta^{-\sigma_3 /2}\frac{1}{\sqrt{2}}\left(\begin{matrix}
    1 & 1\\
    1 & -1
    \end{matrix}\right)\left(I+\left(\begin{matrix}
    \frac{\nu(\nu+1)}{2\zeta^2} & \frac{\nu}{\zeta^2}\\
    \frac{\nu+1}{\zeta^2} & \frac{-\nu(\nu+1)}{2\zeta^2}
    \end{matrix}\right)+O(\zeta^{-4})\right)e^{(\frac{\zeta^2}{4}-(\nu +\frac{1}{2})\ln \zeta)\sigma_3}.
    \end{equation}

\end{itemize}

From \eqref{Z-large}, we know that $P^{(r)}(z)$ satisfies the asymptotic behavior as $t\to +\infty$ in the form of
\begin{equation} \label{prz-large-t}
P^{(r)}(z)= \left(\begin{matrix}
 1 & -\frac{\nu s_3}{h_1}e^{\frac{2it}{3}} \beta^2(z)\frac{1}{\sqrt{t}\zeta}\\
-\frac{h_1}{s_3}e^{-\frac{2it}{3}} \beta^{-2}(z)\frac{1}{\sqrt{t}\zeta} & 1
\end{matrix} \right)\times (I+O(1/t))N(z),
\end{equation}
which holds uniformly as $z \in \partial U(z_+, \delta) $.

Recall that in \cite{Dai:Hu2017}, $\nu= -\frac{1}{2\pi i} \ln(1-s_1 s_3)$ is purely imaginary under the condition $|s_1|=|s_3|<1$. Then, we have, as $t \to \infty$,
$\beta^{\pm 2}(z)\frac{1}{\sqrt{t}\zeta}=O(t^{-1/2})$ and $P^{(r)}(z)N^{-1}(z)\to I.$
But in the present case $|s_1|=|s_3|>1$,
\begin{equation}
\nu= -\frac{1}{2\pi i} \ln(1-s_1 s_3)\equiv \nu_0 - \frac{1}{2}, \qquad \textrm{with } \nu_{0}\in i\mathbb{R},
\end{equation}
is not purely imaginary any more. This change leads to that
\begin{equation}
P^{(r)}(z)N^{-1}(z)\not\to I, \qquad \textrm{as } t \to +\infty.
\end{equation}
Precisely, we denote
\begin{equation}\label{tilde-alpha-beta}
\tilde{\beta}(z)=\left(\zeta(z)\frac{z+1/2}{z-1/2}\right)^{2\nu_0 }\quad \textrm{and} \quad \tilde{\a}(z)=\frac{z+1/2}{z-1/2}.
\end{equation}
Then, we have, in $U(z_{+}, \delta)$, a neighborhood of $z_{+}= \frac{1}{2}$
\begin{equation}
\beta^2(z)\sqrt{t}\zeta(z) =\frac{\tilde{\beta}(z)}{\tilde{\a}(z)}= O(1), \qquad \textrm{as } t \to +\infty.
\end{equation}
This gives us
\begin{eqnarray}\label{Pr-Er-rel}
P^{(r)}(z)&=& \left(\begin{matrix}
 1 & -\frac{\nu s_3}{h_1}e^{\frac{2it}{3}} \frac{\tilde{\beta(z)}}{\tilde{\a}(z)}\frac{1}{t\zeta^2(z)}\\
 -\frac{h_1}{s_3}e^{-\frac{2it}{3}} \frac{\tilde{\a}(z)}{\tilde{\beta}(z)} & 1
\end{matrix} \right)\times (I+O(1/t))N(z)\nonumber\\
&=&\left(\begin{matrix}
 1 & 0\\
-\frac{h_1}{s_3}e^{-\frac{2it}{3}} \frac{\tilde{\a}(z)}{\tilde{\beta}(z)} & 1
\end{matrix} \right)\times (I+O(1/t))N(z)
\end{eqnarray}
Finally, we construct $Q^{(r)}(z)$ in the form:
\begin{equation}\label{Qr-def}
Q^{(r)}(z)=E_r(z)P^{(r)}(z),
\end{equation}
where $P^{(r)}(z)$ is defined in \eqref{Pr-def} and $E_r(z)$ has a simple pole at $z=\frac{1}{2}$ with
\begin{equation}\label{Er-def}
E_r(z)=\left(\begin{matrix}
 1 & 0\\
\frac{h_1}{s_3}e^{-\frac{2it}{3}} \frac{\tilde{\a}(z)}{\tilde{\beta}(z)} & 1
\end{matrix}\right),
\end{equation}
and $\tilde{\a}(z)$ and $\tilde{\beta}(z)$ defined in \eqref{tilde-alpha-beta}.

Upon the RH problem of $P^{(r)}(z)$ in \cite{Dai:Hu2017}, it is easy to check that $Q^{(r)}(z)$ satisfies exactly the RH problem stated at the beginning.

At $z=z_{-} = - 1/2$, due to the symmetry property of $T(z)$ depicted in \ref{contour(LD)}, the parametrix $Q^{(l)}(z)$ can be constructed similarly by
\begin{equation} \label{Qrl-realtion}
Q^{(l)}(z)=\sigma_2 Q^{(r)}(-z) \sigma_2.
\end{equation}

\subsection{Local parametrix near the origin}\label{Local parametrix near the origin}\label{sec-paramx-0}

In this section, we will seek a parametrix $Q^{(0)} (z)$ in $U(0, \delta)$ which solves the RH problem as follows:
\begin{itemize}
  \item[(a)] $Q^{(0)}(z)$ is analytic when $z \in U(0, \delta) \setminus \Sigma_T$;

  \item[(b)] When $z\in U(0, \delta)\cap \Sigma_T$, $Q^{(0)}(z)$ satisfies the same jumps as $T(z)$; see Figure \ref{contour(LD)};

  \item[(c)] On $\partial U(0, \delta) = \{ z\in \mathbb{C} : \ |z| = \delta \}$, we have
  \begin{equation} \label{psi0-psid-match}
    Q^{(0)}(z) [N(z) ]^{-1} = I + O (1/t) \qquad \textrm{as } t \to + \infty;
  \end{equation}

  \item[(d)] $Q^{(0)}(z)$ has the following singularities at $z =0$, :
  \begin{eqnarray}\label{U-origin-asy}
  Q^{(0)}(z)\left(\begin{matrix}
  1 & \frac{s_3 +i e^{-\pi i \a}}{1-s_1 s_3}\\
   0 & 1
  \end{matrix}\right) = O\left(\begin{matrix}
  |z|^{\a} &  |z|^{-\a}\\
   |z|^{\a} & |z|^{-\a}
  \end{matrix}\right), \quad \im{z}>0,\\
  Q^{(0)}(z)\left(\begin{matrix}
  1 & 0\\
   -\frac{s_3 +i e^{-\pi i \a}}{1-s_1 s_3} & 1
  \end{matrix}\right) = O\left(\begin{matrix}
  |z|^{\a} &  |z|^{-\a}\\
   |z|^{\a} & |z|^{-\a}
  \end{matrix}\right), \quad \im{z}<0.
  \end{eqnarray}
\end{itemize}

In \cite{Dai:Hu2017}, we just obtained the existence of the solution of the above RH  problem by proving a vanishing lemma. Fortunately, due to the model RH problem for $M(\eta)$ and its solution given in \eqref{M-expression}, we can construct the parametrix at origin explicitly right now.

To construct $Q^{(0)}(z)$ with the use of $M(\eta)$, we define the following conformal mapping in $U(0, \delta)$:
\begin{equation}\label{map-z-eta}
\eta(z):= i \tilde{\theta}(z) = z - \frac{4}{3} z^3.
\end{equation}
Then the parametrix near origin is given by
\begin{equation} \label{p0z-defn}
Q^{(0)}(z)=E(z)M(t \eta(z)) e^{-it \eta(z) \sigma_3} \begin{cases}\D
 e^{-\pi i \nu \sigma_3}, & \im z >0,\\
 e^{\pi i \nu \sigma_3}, & \im z <0,
 \end{cases}
\end{equation}
with $E(z)$ analytic in $U(0, \delta)$ and defined by
\begin{equation}
E(z):=\left(\frac{z+1/2}{1/2 -z}\right)^{\nu \sigma_3}.
\end{equation}
Here we choose the branch such that $\arg(1/2\pm z)\in(-\pi, \pi)$, then $E(z)$ satisfies the following property:
 \begin{equation}
 E(z) = \left(\frac{z+1/2}{z-1/2}\right)^{\nu \sigma_3}
 \begin{cases}\D
 e^{\pi i \nu \sigma_3}, & \im z >0,\\
 e^{-\pi i \nu \sigma_3}, & \im z <0.
 \end{cases}
 \end{equation}
Combining with the model RH problem for $M$, we can easily check that $Q^{(0)}(z)$ defined in \eqref{p0z-defn} solves the RH problem stated at the beginning. As a consequence, we have finished the parametrix construction at the origin.

\subsection{Error Estimation}

Now we define an error function as follows:
\begin{equation}\label{error}
R(z)=\begin{cases}\D T(z)[(Q^{(r)})(z)]^{-1}, & \quad z \in U(z_+,\delta)\setminus \Sigma_T;\\
T(z)[(Q^{(l)})(z)]^{-1}, & \quad z \in U(z_-,\delta)\setminus \Sigma_T; \\
T(z)[(Q^{(o)})(z)]^{-1}, & \quad z \in U(0,\delta)\setminus \Sigma_T; \\
T(z) [N(z)]^{-1}, &\quad elsewhere.
\end{cases}
\end{equation}
Then, $R(z)$ satisfies the RH problem depicted in Figure \ref{contour(error)}:
\begin{itemize}
\item[(a)] $R(z)$ is analytic when $ z \in \mathbb{C}\cut \Sigma_R$; see Figure \ref{contour(error)}.
\begin{figure}[h]
  \centering
  \includegraphics[width=12cm]{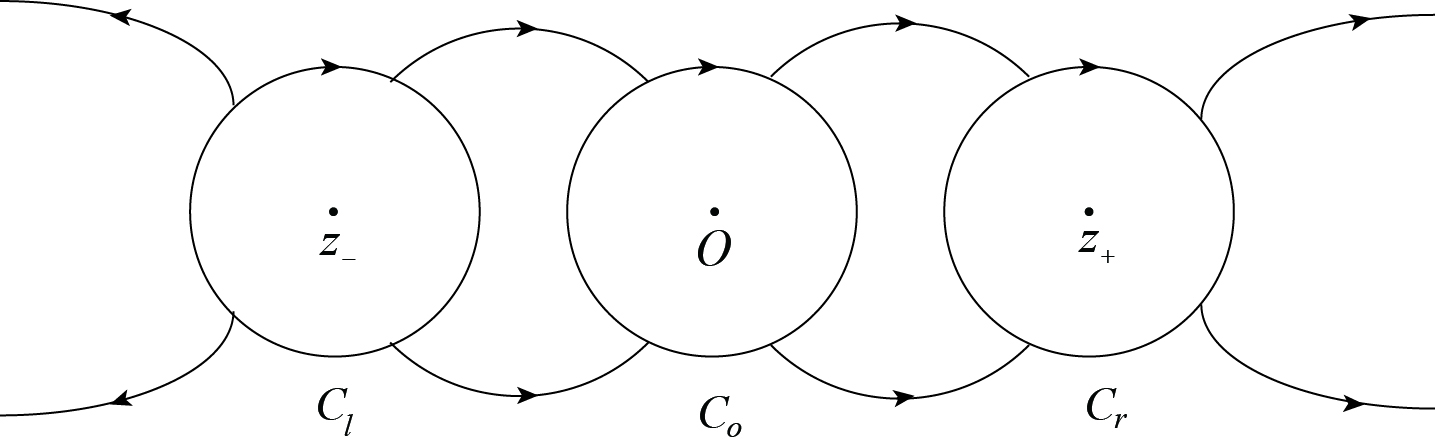}
  \caption{The contour $\Sigma_R$. Here $C_l$, $C_r$ and $C_0$ denote $\partial U(\{z_{\pm},0\}, \delta)$, respectively.}\label{contour(error)}
\end{figure}

\item[(b)] On $\Sigma_R$, $
R_{+}(z)=R_{-}(z) J_{R}(z),$
where
\begin{equation}\label{Jr-def}
J_{R}(z)=\begin{cases}\D Q^{(r)}(z) [N(z)]^{-1},& \quad z\in \partial {U}(z_{+}, \delta),\\
Q^{(l)}(z) [N(z)]^{-1},& \quad z\in \partial {U}(z_{-}, \delta),\\
Q^{(0)}(z)[N(z)]^{-1},& \quad z\in \partial {U}(0, \delta),\\
N(z)J_{T}(z) [N(z)]^{-1},& \quad z\in \Sigma_R \cut \{ \partial {U}(z_{\pm}, \delta) \cup \partial {U}(0, \delta) \}.
\end{cases}
\end{equation}
\item[(c)]$R(z)$ has two simple poles at $z_{\pm}=\pm\frac{1}{2}$. Moreover, if we let $R(z)=(R^{(1)}(z), R^{(2)}(z))$ with $R^{(1)}(z)$ and $R^{(2)}(z)$ denoting the corresponding columns of $R(z)$, then we have from \eqref{Qr-def}-\eqref{Qrl-realtion} and \eqref{error} that
\begin{eqnarray}
&&\mathop{\textrm{res}}\limits_{z=z_{+}}R^{(1)}(z) = R^{(2)}(z_{+})\left(-\frac{h_1}{s_3}\frac{e^{-\frac{2it}{3}}}{\tilde{\beta}(z_{+})}\right) \label{res-z+}\\
&&\mathop{\textrm{res}}\limits_{z=z_{-}}R^{(2)}(z) = R^{(1)}(z_{-})\left(-\frac{h_1}{s_3}\frac{e^{-\frac{2it}{3}}}{\tilde{\beta}(-z_{-})}\right)\label{res-z-}.
\end{eqnarray}
\item[(d)]As $z\to \infty$, $R(z)= I+O(1/z)$.
\item[(e)]As $z\to 0$, $R(z)$ is bounded.
\end{itemize}

It is easy to show that $\det R(z)\equiv 1$ via the Liouville's theorem. Consequently, the solution of the above RH problem, if it exists, is unique.

When $ t \to \infty$, from the matching conditions given in \eqref{psir-psid-match} and \eqref{psi0-psid-match} as well as the jump matrices $J_T(z)$ formulated in Figure \ref{contour(LD)}, we obtain
\begin{equation}\label{Jr-t-order}
J_{R}(z)=\begin{cases}\D
I+O(1/t), & \quad z \in \partial U(z_{\pm}, \delta)\cup \partial U(z_{0}, \delta),\\
I+O(e^{-c_1 t}), & \quad z\in \Sigma_R \cut \partial {U}(z_{\pm}, \delta) \cup \partial {U}(0, \delta),
\end{cases}
\end{equation}
with $c_1$ a positive constant.

However, since $R(z)$ has two simple poles at $z=z_{\pm}$, we can not find the integral equation which is equivalent to the RH problem of $R$ like the one in \cite{Dai:Hu2017}. Indeed, we can deal with this kind of RH problems by simplifying the problem to the one without poles via a certain ``dressing" procedure; see \cite{Bothner:Its2012}.

We put \begin{equation}\label{D-W-def}
R(z)=(z I + D)W(z)\left(\begin{matrix}
\frac{1}{z-z_{+}} & 0\\
0 & \frac{1}{z-z_{-}}
\end{matrix}\right),
\end{equation}
where $D \in \mathbb{C}^{2 \times 2}$  is constant to be determined. Then $W(z)$ satisfies the following RH problem:
\begin{itemize}
\item[(a)] $W(z)$ is analytic when $z\in \mathbb{C}\cut \Sigma_R$.
\item[(b)] On $\Sigma_R$, $W(z)$ satisfies the jump conditions:
\begin{equation}
W_{+}(z)=W_{-}(z)J_{W}(z), \quad z\in \Sigma_R,
\end{equation}
where
\begin{equation}
J_{W}(z)=\left(\begin{matrix}
\frac{1}{z-z_{+}} & 0\\
0 & \frac{1}{z-z_{-}}
\end{matrix}\right)J_{R}(z)\left(\begin{matrix}
z-z_{+} & 0\\
0 & z-z_{-}
\end{matrix}\right).
\end{equation}
\item[(c)] $W(z)$ is bounded at $z=z_{\pm}$.
\item[(d)] As $z\to \infty$, $W(z)$ satisfies the asymptotic behavior:
\begin{equation}
W(z)= I+O(1/z), \quad z\to \infty.
\end{equation}
\end{itemize}

From \eqref{Jr-t-order}, we obtain that the jump matrices $J_{W}(z)$ tend to unit matrix uniformly as $t\to \infty$ and they satisfy the same asymptotics as $J_{R}$ in \eqref{Jr-t-order}.
Therefore, the RH problem of $W$ is equivalent to the singular integral equation
\begin{equation}\label{W-Cau-int}
W_{-}(z)=I +\frac{1}{2\pi i}\int_{\Sigma_{R}}W_{-}(z')(J_{W}(z')-I)\frac{\textrm{d} z'}{z' -z}, \quad z \in \Sigma_{R}.
\end{equation}

According to the standard arguments with respect to the above Cauchy operator, we know that the above operator is contracting for sufficiently large $t$, which indicates that the equation \eqref{W-Cau-int} can be solved in $L^2(\Sigma_R)$ by iterations; see \cite{Deift1999,Dei:Kri:McL:Ven:Zhou1999-2}. Also, based on the above mentioned arguments, as $t \to \infty$, we have
\begin{equation}\label{est-W}
W(z) =I + O(1/t), \qquad \textrm{as } z \in \mathbb{C} \setminus \Sigma_R.
\end{equation}

Finally, we are left with the constant matrix $D$ to be determined. From the residue conditions \eqref{res-z+} and \eqref{res-z-}, we have
\begin{equation}
\mathop{\textrm{res}}_{z=z_{+}}R^{(1)}(z) =(z_{+}I+D) W^{(1)}(z_{+})=(z_{+}I+D) W^{(2)}(z_{+})\left(-\frac{h_1}{s_3}\frac{e^{-\frac{2it}{3}}}{\tilde{\beta}(z_{+})}\right),
\end{equation}
\begin{equation}
\mathop{\textrm{res}}_{z=z_{-}}R^{(2)}(z) = (z_{-}I+D) W^{(2)}(z_{-})=(z_{-}I+D) W^{(1)}(z_{-})\left(\frac{h_1}{s_3}\frac{e^{-\frac{2it}{3}}}{\tilde{\beta}(z_{+})}\right),
\end{equation}
which implies
\begin{equation}\label{D-expression}
D=\left(W(z_{+})\left(\begin{matrix}
1\\
p
\end{matrix}\right), W(z_{-})\left(\begin{matrix}
-p\\
1
\end{matrix}\right)\right)\left(\begin{matrix}
z_{-} & 0\\
0 & z_{+}
\end{matrix}\right)\left(W(z_{+})\left(\begin{matrix}
1\\
p
\end{matrix}\right), W(z_{-})\left(\begin{matrix}
-p\\
1
\end{matrix}\right)\right)^{-1}
\end{equation}
with
\begin{equation}\label{p-def}
 p=\frac{h_1}{s_3}\frac{e^{-\frac{2it}{3}}}{\tilde{\beta}(z_{+})}.
\end{equation}
From \eqref{W-Cau-int} and the definitions of the jumps in the RH problem of $W$, we obtain that
\begin{equation}\label{asy-W-z-pm}
W(z_{\pm}) = I +O(1/t),\quad t\to \infty.
\end{equation}
Hence, the left matrix in \eqref{D-expression} is invertible, and then the right matrix inverse exists, for sufficiently large $t$ except those such that
\begin{equation}\label{p-condition}
1+p^2 = 0.
\end{equation}

According to the definitions of the parameters $p$ in \eqref{p-def} and $h_1$ in \eqref{h1-def}, and the property of Gamma function that
\begin{equation}
\Gamma(\nu)\Gamma(1-\nu)=\frac{\pi}{\sin {\pi\nu}},
\end{equation}
we have $|p|=1$ and
\begin{equation}\label{p-arg}
p=-i\frac{s_1 \Gamma(\nu+1)}{\sqrt{2\pi(|s_1|^2 -1)}}e^{-\frac{2it}{3}-\nu_0 \ln 8t-\frac{\pi i \nu_0}{2}}=-ie^{-\frac{2it}{3}-\nu_0 \ln 8t + i \arg \Gamma(\nu + 1) + i \arg s_1}.
\end{equation}
Then \eqref{p-condition} yields a family of points $\{t_n\}$ by the following equation
\begin{equation}
\frac{2t_n}{3}+ \frac{\ln(|s_1|^2 - 1)}{2\pi}\ln(8t_n)-\arg\Gamma(\nu+1)-\arg s_1=n\pi, \quad n=0,1,2,...
\end{equation}
The above $\{t_n\}$ will finally lead to the appearance of the singularities in the asymptotics \eqref{asy-neg}.

\section{Proof of the main result}\label{sec-proof}
From now on, we will calculate the asymptotics \eqref{asy-neg} of $u(x;\a)$ by taking $t$ bounded away from the neighborhood of the points $t_n$ given in last section. Tracing back the transformation $W \to R \to T \to U \to \Psi_\a$ and using the connection that $u(x;\alpha)=2\biggl(\Psi_{1}(x)\biggl)_{12} $, an expression of $u(x;\a)$ in terms of the integral equation \eqref{W-Cau-int} is obtained as follows:
\begin{equation}
u(x;\a)=2\sqrt{-x}D_{12}-\frac{\sqrt{-x}}{\pi i} \int_{\Sigma_R} \biggl(W_{-}(z')\left(J_W(z')-I\right) \biggr)_{12} \textrm{d}z'.
\end{equation}
Then neglecting the exponentially small contributions on $\Sigma_R \cut \partial {U}(z_{\pm}, \delta) \cup \partial {U}(0, \delta)$ and combining \eqref{est-W} and \eqref{Jr-t-order} yield that,
\begin{eqnarray}\label{u-D-relation}
u(x;\a)&=&2\sqrt{-x}D_{12}-\frac{\sqrt{-x}}{\pi i} \int_{C_l \cup C_o \cup C_r} (J_{W}(z))_{12} \textrm{d}z' +O((-x)^{-5/2})\nonumber \\
&=& 2\sqrt{-x}D_{12} +O((-x)^{-1}),\qquad \textrm{as}\,\, x\to -\infty.
\end{eqnarray}

Now we only need to compute the term $D_{12}$. From \eqref{D-expression}, combing with \eqref{W-Cau-int} and \eqref{est-W}, we derive that
\begin{equation}\label{D-12}
(D_{12})^{-1}=\frac{W_{22}(z_{-})-pW_{21}(z_{-})}{W_{12}(z_{-})-pW_{11}(z_{-})}-\frac{W_{21}(z_{+})+pW_{22}(z_{+})}{W_{11}(z_{+})+pW_{12}(z_{+})},
\end{equation}
with $p$ be defined in \eqref{p-def}. As $x\to -\infty$, the behavior of $W(z)$ in \eqref{asy-W-z-pm} gives us that
\begin{equation}
W_{11}(z_{\pm}) = 1+O((-x)^{-\frac{3}{2}}), \quad W_{22}(z_{\pm}) = 1+O((-x)^{-\frac{3}{2}})
\end{equation}
and
\begin{equation}
W_{12}(z_{\pm}) = O((-x)^{-\frac{3}{2}}), \quad W_{21}(z_{\pm}) = O((-x)^{-\frac{3}{2}}).
\end{equation}
The above three formulas imply that
\begin{equation}\label{D12-p-relation}
(D_{12})^{-1}= -2\re{p} + O((-x)^{-\frac{3}{2}}), \qquad \textrm{as}\,\,  x\to -\infty.
\end{equation}
Then taking \eqref{p-arg} into account, \eqref{u-D-relation} turns into
\begin{equation}
u(x;\a)=\frac{\sqrt{-x}}{\sin\{\frac{2}{3}(-x)^{\frac{3}{2}}+\frac{3}{4}d^2\ln(-x)+\phi\}+O((-x)^{-\frac{3}{2}})}+O((-x)^{-1}),\qquad \textrm{as}\,\, x\to \infty,
\end{equation}
with
\begin{equation}
d = \frac{1}{\sqrt{\pi}}\sqrt{\ln(|s_1|^2-1)}
\end{equation}
and
\begin{equation}
\phi = \frac{3\ln 2}{2} d^2  -\arg \Gamma{\biggr(\frac{1}{2}+\frac{1}{2}id^2}\biggr)-\arg (s_1).
\end{equation}
Substituting the condition \eqref{stokes-real-as}, we deduce the asymptotics and connection formulas in Theorem \ref{main-thm}.

\section*{Acknowledgements}
 The author is grateful to Dan Dai for useful and stimulation discussions. The author also thanks anonymous referees for their valuable comments and suggestions, which
improved the article a lot.

The work is partially supported by the Research Grants Council of the Hong
Kong Special Administrative Region, China (Project No. CityU 11300115, CityU 11303016) and National Natural Science Foundation of China (Grant No. 11871345).

\end{document}